\documentclass[11pt]{amsart}

\usepackage{amsmath,amssymb,mathrsfs,amstext,amscd,amsfonts,setspace,verbatim}
\usepackage[pdftex]{epsfig,color,graphics}
\usepackage{graphics,color}
\usepackage{latexsym}
\usepackage{paralist}
\usepackage{fullpage}                           
\usepackage{hyperref}
\usepackage{ifpdf}
\pdfoutput=1
\providecommand{\doi}[1]{\href{http://dx.doi.org/#1}{doi:#1}}
\providecommand{\arxiv}[2][]{\href{http://www.arXiv.org/abs/#2}{arXiv:#2}}

\newtheorem{proposition}{Proposition} 
\newtheorem{lemma}{Lemma} 
\newtheorem{theorem}{Theorem} 

\theoremstyle{remark}
\newtheorem{remark}{Remark} 
\newtheorem{example}{Example}
\title{Minimal half-spaces and external representation of tropical polyhedra}
\author{{S}t{\'e}phane {G}aubert}
\address{St\'ephane Gaubert, INRIA and Centre de Math\'ematiques Appliqu\'ees (CMAP), \'Ecole Polytechnique. Postal address: CMAP, \'Ecole Polytechnique, 91128 Palaiseau Cedex France.}
\email{Stephane.Gaubert@inria.fr}
\author{Ricardo D. Katz}
\address{Ricardo D. Katz, CONICET. Postal address: Instituto de Matem\'atica ``Beppo Levi'', Universidad Nacional de 
Rosario, Avenida Pellegrini 250, 2000 Rosario, Argentina.}
\email{rkatz@fceia.unr.edu.ar}
\thanks{The first author was partially supported by Arpege programme of the French National Agency of Research (ANR), project ``ASOPT'', number ANR-08-SEGI-005 and by the Digiteo project DIM08 ``PASO'' number 3389}

\DeclareMathAlphabet{\mathbbold}{U}{bbold}{m}{n}
\newcommand{\zero}{\mathbbold{0}}
\newcommand{\unit}{\mathbbold{1}}
\newcommand{\new}[1]{{\em #1}}
\newcommand{\myminus}[1]{{#1}^{-}}
\newcommand{\set}[2]{\left\{#1\mid\,#2\right\}}
\newcommand{\R}{\mathbb{R}}
\newcommand{\N}{\mathbb{N}}

\newcommand{\Rmin}{\R\cup\{+\infty\}}
\newcommand{\rmax}{\R_{\max}}
\newcommand{\rmin}{\R_{\min}}
\newcommand{\co}{\mbox{\rm co}\,}
\newcommand{\cone}{\mbox{\rm cone}\,}

\newcommand{\supp}{\mbox{\rm supp}\,}

\newcommand{\uvector}{\mbox{\rm e}}

\newcommand{\sop}{\mbox{\rm supp}\,}
\newcommand{\type}{\mbox{\rm type}}
\newcommand{\sV}{\mathscr{V}}
\newcommand{\sC}{\mathscr{C}}
\newcommand{\sW}{\mathscr{W}}

\newcommand{\sH}{\mathscr{H}}
\newcommand{\sX}{\mathscr{X}}
\newcommand{\sY}{\mathscr{Y}}
\newcommand{\sZ}{\mathscr{Z}}

\begin{document}
\begin{abstract}
We give a characterization of the minimal tropical half-spaces containing a given tropical polyhedron, from which we derive a counter example showing that the number of
such minimal half-spaces can be infinite, contradicting some
statements which appeared in the tropical literature, and disproving
a conjecture of F. Block and J. Yu. 
We also establish an analogue of
the Minkowski-Weyl theorem, showing that a tropical polyhedron
can be equivalently represented internally (in terms of extreme points and rays)
or externally (in terms of half-spaces containing it). A canonical 
external representation of a polyhedron turns out to be provided by
the extreme elements of its tropical polar. We characterize these
extreme elements, showing in particular that they are determined
by support vectors.
\end{abstract}

\maketitle 

\section{Introduction}
Max-plus or tropical convexity has been developed by several researchers
under different names, with various motivations.
It goes back at least to the work of
Zimmermann~\cite{zimmerman77}. It was studied by Litvinov, Maslov, and 
Shpiz~\cite{litvinov00}, in relation to problems of calculus of variations,
and by Cohen, Gaubert, and Quadrat~\cite{cgq00,cgq02}, motivated by discrete event system
problems (max-plus polyhedra represent invariant spaces
of max-plus linear dynamical systems~\cite{ccggq99}). Some of this work was pursued with Singer (see~\cite{cgqs04}), 
with motivations from generalized convexity~\cite{ACA}. 
The work of Briec and Horvath~\cite{BriecHorvath04} is also in the setting
of generalized convexity. Develin and Sturmfels~\cite{DS04}
pointed out some remarkable relations with tropical geometry, and 
developed a new approach, thinking of tropical polyhedra
as polyhedral complexes in the usual sense. This was the starting
point of several works of the same authors, of Joswig~\cite{joswig04}
and of Block and Yu~\cite{blockyu06}.
Some of the previously mentioned researchers, and some other ones,
including Allamigeon, Butkovi\v{c}, Goubault, Katz, Nitica, Meunier, Sergeev, Schneider, have recently made a number of works in the field, we refer
the reader to~\cite{AGK09,BSS,relmics,gk07,NiticaSinger07a,JSY07,joswig-2008,katz08,GM08,goubault} for a representative set of contributions.

A closed convex set can be represented classically in two different ways,
either internally, in terms of extreme elements
(extreme points and rays) and lineality space, or externally,
as the intersection of (closed) half-spaces.

The max-plus or tropical analogue of the external representation,
specially in the case of polyhedra, is the main object of this paper.

The existence of an external representation relies
on separation arguments. In the max-plus setting, several
separations theorems have been obtained, with various degrees of generality
and precision, by Zimmermann~\cite{zimmerman77}, by Samborski{\u\i} and Shpiz~\cite{shpiz}, by Cohen, Gaubert, and Quadrat~\cite{cgq00,cgq02} with a further
refinement in a work with Singer~\cite{cgqs04}, and by Develin and Sturmfels~\cite{DS04}. In particular,
the results of~\cite{cgq00,cgq02,cgqs04}
yield a simple geometric construction of the separating
half-space, showing the analogy with the Hilbert space case. This
geometric approach was extended to the case of the separation of several
convex sets by Gaubert and Sergeev~\cite{gauser}, using a cyclic projection
method. Briec and Horvath derived a separation theorem
for two convex sets using a different approach~\cite{briechorvath08a}.

The existence of an internal representation relies on
Krein-Milman type theorems. Results of this kind were
established by Butkovi\v{c}, Schneider, and Sergeev~\cite{BSS}
and by the authors~\cite{gk07}, who also studied in~\cite{katz08}
the analogue of the polar of a convex
set, which consists of the set of inequalities satisfied
by its elements.

Polyhedra are usually defined by the condition that
they have a finite external or internal representation,
the equivalence of both conditions being the classical
Minkowski-Weyl theorem.

In the max-plus setting, a first result of this nature
was established by Gaubert in~\cite[Ch.~III, Th.~1.2.2]{gaubert92a},
who showed that a finitely generated
max-plus cone can be characterized by finitely many max-plus linear inequalities.
One element of the proof is an argument showing that the set
of solutions of a system of max-plus linear equations is finitely generated,
an observation which was already made by Butkovi\v{c} and Hegedus~\cite{butkovicH}. Some accounts in English of the result of~\cite{gaubert92a}
appeared in~\cite{maxplus97,abg05,relmics}. 

To address the same issue, 
Joswig~\cite{joswig04} introduced the very interesting notion
of minimal half-spaces (tropical half-spaces
that are minimal for inclusion among the ones containing a given tropical
polytope), he stated that the apex of such a tropical half-space
is a vertex of the classical polyhedral complex arising from the polytope,
and deduced that there are only finitely many such minimal half-spaces.
This statement was refined by a conjecture of Block and Yu, 
characterizing the minimal half-spaces, in the generic
case~\cite[Conj.~14]{blockyu06}. 

The finiteness of the
number of minimal half-spaces is an appealing property,
which is geometrically quite obvious in dimension 2.
It came to us as a surprise that it does not hold
in higher dimensions.
We give here a counter example (Example~\ref{Ejemplo1} below),
contradicting the finiteness of the number of minimal half-spaces
containing a tropical polyhedron
(Corollary~3.4 of~\cite{joswig04}) and disproving Conjecture~14 of~\cite{blockyu06}.
The analysis of the present counter example is based on
a general characterization of the minimal half-spaces, Theorem~\ref{TheoCharacMin} below, the main result of this paper, which gives some answer
to the question at the origin of the conjecture of Block and Yu.

In a preliminary section, we establish an analogue
of the Minkowski-Weyl theorem
(Theorem~\ref{th-mw}), showing that a tropical polyhedron can be
equivalently described either as the sum of the convex hull 
of finitely many points and of the cone generated by finitely
many vectors, or as the intersection of finitely many half-spaces
(there is no tropical analogue of the lineality space).
The proof is based on the idea of~\cite{gaubert92a}
(Theorem~\ref{MinkowskiWeylCones} below), which is combined
with the results of~\cite{gk07}.

In the final section, we characterize the extreme elements
of the polar of a tropical polyhedral cone. The set of extreme
elements of this polar has the property that any inequality satisfied
by all the elements of the cone is a max-plus linear combination of
the inequalities represented by these extreme elements,
and it is the unique minimal set with this property (up to a scaling).
In particular, these extreme elements provide a finite representation
of the original polyhedron as the intersection of half-spaces.
Theorem~\ref{TheoCharacExtreme} below characterizes these extreme elements,
showing in particular that each of them is determined by support vectors.

\section{The Tropical Minkowski-Weyl Theorem}\label{sec-tmw}
Let us first recall some basic definitions.
The max-plus semiring, $\rmax$, 
is the set $\R\cup\{-\infty\}$ equipped with the 
addition $(a,b)\mapsto \max(a,b)$ and the multiplication $(a,b)\mapsto
a+b$. To emphasize the semiring structure, we write $a\oplus b:=\max(a,b)$,
$ab:=a+b$, $\zero:=-\infty$ and $\unit:=0$.  The term ``tropical''
is now used essentially as a synonym of max-plus. 
The semiring operations are extended in 
the natural way to matrices over the max-plus semiring: 
$(A\oplus B)_{ij}:=A_{ij}\oplus B_{ij}$, 
$(AB)_{ij}:=\oplus_k A_{ik} B_{kj}$ and $(\lambda A)_{ij}:=\lambda A_{ij}$ 
for all $i,j$, 
where $A,B$ are matrices of compatible sizes and $\lambda \in \rmax$. 
We denote by $\uvector^k\in \rmax^n$ the $k$-th unit vector, 
i.e.\ the vector defined by: $(\uvector^k)_k:=\unit $ and 
$(\uvector^k)_h:=\zero $ if $h \neq k$. 

We consider $\rmax $ equipped with the 
usual topology (resp. order), which can be defined by the metric: 
$d(a,b):=|\exp(a)-\exp(b)|$. The set $\rmax^n$ is equipped with 
the product topology (resp. order). 
Note that the semiring operations are continuous with 
respect to this topology.  

A subset $\sV$ of $\rmax^n$ is said to be a 
\new{max-plus or tropical cone} if it is 
stable by max-plus linear combinations, meaning that
\begin{align}
\lambda u\oplus \mu v \in \sV
\label{e-stable}
\end{align}
for all $u,v\in \sV$ and $\lambda,\mu \in \rmax$. Note that, 
in the max-plus setting, positivity constraints are implicit because 
any scalar $\lambda \in \rmax$ satisfies $\lambda \geq \zero$.  
As a consequence, max-plus cones turn out to share many 
properties with classical convex cones. 
This analogy leads to define \new{max-plus convex subsets} $\sC$ of 
$\rmax^n$ by requiring them to be stable by max-plus convex combinations, 
meaning that $\lambda u\oplus \mu v \in \sC$ holds for all $u,v\in \sC$
and $\lambda,\mu\in \rmax$ such that $\lambda\oplus \mu=\unit$. 
We denote by $\cone(\sX)$ the smallest cone containing a
subset $\sX$ of $\rmax^n$, and by $\co(\sX)$ the smallest  
convex set containing it. Therefore, $\cone(\sX)$ (resp. $\co(\sX)$) 
is the set of all max-plus linear (resp. convex) 
combinations of finitely many elements of $\sX$. 
A cone $\sV$ is said to be finitely generated 
if there exists a finite set $\sX$ such that $\sV=\cone(\sX)$, 
which equivalently means that $\sV=\set{Cw}{w\in \rmax^t}$ 
for some matrix $C\in \rmax^{n\times t}$. 

A \new{half-space} of $\rmax^n$ is a set of the form
\[
\sH=\set{x\in \rmax^n}{\oplus_{1\leq i\leq n} a_i x_i\leq \oplus_{1\leq j\leq n} b_j x_j} \; , 
\]
where $a,b\in \rmax^n$, and an \new{affine half-space} 
of $\rmax^n$ is a set of the form
\[
\sH=\set{x\in \rmax^n}{\left( \oplus_{1\leq i\leq n} a_i x_i\right) \oplus c\leq 
\left( \oplus_{1\leq j\leq n} b_j x_j\right) \oplus d} \; , 
\]
where $a,b\in \rmax^n$ and $c,d\in \rmax$. With the classical notation,
the latter set can be written as
\[
\sH=\big\{x\in \rmax^n\mid \max\big( \max_{1\leq i\leq n} a_i +x_i,c\big)\leq 
\max\big( \max_{1\leq j\leq n} b_j +x_j, d\big) \big\} \enspace .
\]
Note that half-spaces are max-plus cones. 

Classical polyhedra can be defined either 
as a finite intersection of affine half-spaces, 
or in terms of finite sets of vertices and rays, i.e.\ 
as the Minkowski sum of a polytope and a finitely generated cone. 
Here we adopt the first approach and define a 
\new{max-plus or tropical polyhedron} 
as the intersection of finitely many affine half-spaces. 
We warn the reader that our notion of polyhedra is more general than 
the one used in~\cite{DS04} (the latter reference deals with max-plus
cones having a finite generating family consisting of vectors
with finite entries). 

The following ``conic'' form of the Minkowski-Weyl theorem
is equivalent to a result established in~\cite{gaubert92a},
showing that a finitely generated max-plus cone
is characterized by finitely many max-plus linear equalities. This result
was reproduced (but without its proof) in the 
surveys~\cite{maxplus97,abg05,relmics}.
For the convenience of the reader, we include the proof here.
The ``if'' part is equivalent to the existence of a finite
set of generators of a system of max-plus linear equations, which was
first shown in~\cite{butkovicH}. There have been recently
progresses on these issues, leading to a faster algorithm,
see~\cite{goubault,AGG10}.

\begin{theorem}[Compare with~{\cite[Ch.~III, Th.~1.2.2]{gaubert92a}} and~{\cite[Th.~9]{maxplus97}}]\label{MinkowskiWeylCones}
A max-plus cone is finitely generated if, and only if, it is the intersection 
of finitely many half-spaces. 
\end{theorem}

\begin{proof}
Let $\sV\subset \rmax^n$ be an intersection of $p$ half-spaces. 
We next prove that $\sV $ is finitely generated by induction on $p$. 

When $p=1$, as $\sV=\set{x\in \rmax^n}{\oplus_{1\leq i\leq n} a_i x_i\leq \oplus_{1\leq j\leq n} b_j x_j}= 
\cup_{1\leq j\leq n}\sV_j$, where 
\[ 
\sV_j:=\set{x\in \rmax^n}{a_i x_i\leq b_j x_j,\; \forall i=1,\ldots ,n} \; ,
\] 
to prove that $\sV$ is finitely generated it suffices to show that the cones 
$\sV_j$ are all finitely generated. If $b_j\not =\zero$ and $a_j\leq b_j$, 
then it can be checked that $\sV_j=\cone(\sX_j)$, where 
$\sX_j:=\set{b_j \uvector^i\oplus a_i \uvector^j}{i=1,\ldots ,n}$.  
If $b_j=\zero$ or $a_j>b_j$, then $\sV_j=\cone(\sX_j)$, where 
$\sX_j:=\set{\uvector^i}{a_i=\zero }$. 

Assume now that the intersection of $p$ half-spaces 
is finitely generated and let 
\[
\sV:=\set{x\in \rmax^n}{A x\leq B x}\cap \set{x\in \rmax^n}{a x\leq b x} \; ,
\] 
where $A,B\in \rmax^{p\times n}$ and $a,b\in \rmax^{1\times n}$, 
be an intersection of $p+1$ half-spaces. 
Then, we know that there exists a matrix 
$C\in \rmax^{n\times t}$, for some $t\in \N$, such that 
$\set{x\in \rmax^n}{A x\leq B x}=\set{Cw}{w\in \rmax^t}$. 
As $\sH:=\set{w\in \rmax^t}{aCw\leq bCw}$ is a half-space, 
there exists another matrix $D\in \rmax^{t\times r}$, for some 
$r\in \N$, such that $\sH=\set{Du}{u\in \rmax^r}$. Therefore, 
$\sV= \set{Cw}{aCw\leq bCw}= \set{CDu}{u\in \rmax^r}$ 
is finitely generated.   

Conversely, let $\sV=\set{Cw}{w\in \rmax^t}$, where $C\in \rmax^{n\times t}$, 
be a finitely generated cone. Then, as finitely generated cones are closed 
(see~\cite[Cor.~27]{BSS} or~\cite[Lemma~2.20]{gk07}), it follows from the separation theorem for closed 
cones of~\cite{zimmerman77,shpiz,cgqs04} 
that $\sV$ is the intersection of the half-spaces of $\rmax^n$ 
in which it is contained. Note that a half-space 
$\set{x\in \rmax^n}{a x\leq b x}$ contains $\sV$ if, and only if, 
the row vectors $a$ and $b$ satisfy $aC\leq bC$. 
Since $\set{(a,b)\in \rmax^{1\times 2n}}{aC\leq bC}$ 
is a finite intersection of half-spaces, 
we know by the first part of the proof that there exist matrices 
$A$ and $B$ such that $(a,b)$ satisfies $aC\leq bC$ if, and only if, $(a,b)$ 
is a max-plus linear combination of the rows of the matrix $(A,B)$. 
Therefore, $\sV=\set{x\in \rmax^n}{A x\leq B x}$, i.e.\ 
$\sV$ is an intersection of finitely many half-spaces. 
\end{proof}
Recall that the 
{\em recession cone}~\cite{gk07} of a max-plus convex set $\sC$
consists of the vectors $u$ for which there exists a vector $x\in \sC$ such that
$x\oplus \lambda u\in \sC$ for all $\lambda \in \rmax$. This property is known
to be independent of the choice of $x\in \sC$ as soon as $\sC$ is closed.

Given a max-plus cone $\sV\subset \rmax^n$, a non-zero vector $v \in \sV$ 
is said to be an {\em extreme vector} of $\sV$ 
if the following property is satisfied 
\[
v=u\oplus w,\;u,w\in \sV  \implies v=u \mbox { or } v=w \enspace . 
\] 
The set of scalar multiples of $v$ is an {\em extreme ray} of $\sV$.
Given a max-plus convex set $\sC\subset \rmax^n$, a vector $v\in \sC$
is said to be an {\em extreme point} of $\sC$ if 
\[
v=\lambda u\oplus \mu w,\;u,w\in \sC, \; \lambda,\mu\in \rmax,
\;
\lambda\oplus \mu=\unit  \implies v=u \mbox { or } v=w \enspace . 
\] 

As a corollary of Theorem~\ref{MinkowskiWeylCones} 
we obtain a max-plus analogue of the Minkowski-Weyl theorem,
the first part of which was announced in~\cite{relmics}. 
A picture illustrating the decomposition can be found in~\cite{relmics,gk07}. 
  
\begin{theorem}[Tropical Minkowski-Weyl Theorem]
\label{th-mw}
The max-plus polyhedra are precisely the sets of the form
\[
\co(\sZ)\oplus \cone(\sY) 
\]
where $\sZ,\sY$ are finite sets. 
The set $\cone(\sY)$ in such a representation is unique,
it coincides with the recession cone of the polyhedron.
Any minimal set $\sY$ in such a representation 
can be obtained by selecting precisely one
non-zero vector in each extreme ray of the recession cone
of the polyhedron. 
The minimal set $\sZ$ in such a representation consists of the extreme
points of the polyhedron.
\end{theorem}

Here, $\oplus $ denotes the max-plus Minkowski sum of two subsets, 
which is defined as the set of max-plus sums of a vector 
of the first set and a vector of the second one.

\begin{proof}
Let $\sC\subset \rmax^n$ be a max-plus polyhedron. Then, 
there exist matrices $A,B$ and column vectors $c,d$ such that 
$\sC=\set{x\in \rmax^n}{A x\oplus c\leq B x\oplus d}$. 
Consider the max-plus cone 
\[
\sV:=\set{{x\choose \lambda} \in \rmax^{n+1}}
{A x\oplus c\lambda \leq B x\oplus d\lambda } \; .
\] 
Since $\sV$ is an intersection of finitely many half-spaces, 
by Theorem~\ref{MinkowskiWeylCones} it follows that $\sV=\cone(\sX)$, 
for some finite subset $\sX$ of $\rmax^{n+1}$. 
Note that we can assume, without loss of generality, that 
\begin{eqnarray}\label{DescZXY}
\sX=\set{
\begin{pmatrix}
z \cr 
\unit
\end{pmatrix} 
\in \rmax^{n+1}}{z\in \sZ} \cup
\set{
\begin{pmatrix}
y \cr 
\zero
\end{pmatrix} 
\in \rmax^{n+1}}{y\in \sY}
\end{eqnarray}
for some finite subsets $\sZ,\sY$ of $\rmax^n$. Therefore, we have
\begin{eqnarray*}
x\in \sC \iff 
\begin{pmatrix}
x \cr 
\unit
\end{pmatrix}
\in \sV \iff 
\begin{pmatrix}
x \cr 
\unit
\end{pmatrix}=
\left( \oplus_{z\in \sZ} \lambda_z 
\begin{pmatrix}
z \cr 
\unit
\end{pmatrix}
\right) \oplus \left( \oplus_{y\in \sY} \lambda_y 
\begin{pmatrix}
y \cr 
\zero
\end{pmatrix} 
\right) \iff \\ 
x=\left( \oplus_{z\in \sZ} \lambda_z z\right) \oplus 
\left( \oplus_{y\in \sY} \lambda_y y\right) , \enspace 
\oplus_{z\in \sZ} \lambda_z =\unit \iff 
x\in \co(\sZ) \oplus \cone(\sY) \enspace ,
\end{eqnarray*}
which shows that $\sC=\co(\sZ) \oplus \cone(\sY)$. 

Conversely, let $\sC=\co(\sZ) \oplus \cone(\sY)$, 
where $\sZ,\sY$ are finite subsets of $\rmax^n$. 
Note that $x$ belongs to $\sC$ if, and only if, 
${x\choose \unit }$ belongs to $\sV:=\cone (\sX)$, where $\sX$ 
is the finite subset of $\rmax^{n+1}$ defined in~\eqref{DescZXY}.  
Since $\sV$ is a finitely generated cone, 
we know by Theorem~\ref{MinkowskiWeylCones} 
that there exist matrices $A,B$ and column vectors $c,d$ such that 
$\sV=\set{{x\choose \lambda}\in \rmax^{n+1}}
{A x\oplus c\lambda \leq B x\oplus d\lambda }$. Therefore, 
$\sC=\set{x\in \rmax^n}{A x\oplus c\leq B x\oplus d}$, 
i.e.\ $\sC$ is a max-plus polyhedron.

Now let $\sC=\co(\sZ) \oplus \cone(\sY)$ 
be a max-plus polyhedron. From the definition of recession cones, 
it readily follows that $\cone(\sY)$ 
is contained in the recession cone of $\sC$. 
Assume that $u$ is a vector in the recession cone of $\sC$. 
By the first part of the proof, if we define $\sV:=\cone (\sX)$, 
where $\sX$ is the finite subset of $\rmax^{n+1}$ defined in~\eqref{DescZXY}, 
then there exist matrices $A,B$ and column vectors $c,d$ such that 
$\sV= \set{{x\choose \lambda}\in \rmax^{n+1}}
{A x\oplus c\lambda \leq B x\oplus d\lambda }$ 
and $\sC=\set{x\in \rmax^n}{A x\oplus c\leq B x\oplus d}$. 
Since $u$ is in the recession cone of $\sC$, 
there exists $x\in \sC$ such that $x\oplus \lambda u\in \sC$ 
for all $\lambda \in \rmax$. This means that 
$A (x\oplus \lambda u)\oplus c\leq B (x\oplus \lambda u) \oplus d$
for all $\lambda \in \rmax$, so we conclude that $Au\leq Bu$. 
Therefore, ${u\choose \zero } \in \sV = \cone (\sX)$, 
which implies that $u\in \cone (\sY)$ by the definition of $\sX$.  
In consequence, $\cone (\sY)$ is equal to the recession cone of $\sC$. 

Assume that $z$ is an extreme point of $\sC$. We next show that necessarily 
$z\in\sZ$. To this end, by the definition of extreme points, 
it suffices to show that $z\in \co(\sZ)$. By the contrary, 
assume that $z=x\oplus u$, where $x\in \co(\sZ)$, $u\in \cone (\sY)$ 
and $x_i < u_i$ for some $i\in \{1,\ldots ,n\}$. Then, given 
any non-zero scalar $\lambda < \unit$, we have 
$z=(x\oplus \lambda u)\oplus \lambda (x\oplus (-\lambda )u)$, 
which contradicts the fact that $z$ is an extreme point of $\sC$ 
because $x\oplus \lambda u$ and $x\oplus (-\lambda )u$ 
are two elements of $\sC$ different from $z$. 
Therefore, any extreme point of $\sC$ must belong to $\sZ$. 

Now let $y$ be an extreme vector of the recession cone of $\sC$. 
Since the recession cone of $\sC$ is equal to $\cone (\sY)$, 
by the definition of extreme vectors, 
it follows that a non-zero scalar multiple of $y$ must belong to $\sY$. 

Finally, since $\sC$ is closed because it is a finite intersection of closed sets, 
from Theorem~3.3 of~\cite{gk07} it follows that any minimal 
set $\sY$ in the representation $\sC=\co(\sZ) \oplus \cone(\sY)$ 
can be obtained by selecting precisely one non-zero vector in each extreme ray 
of the recession cone of $\sC$, and a minimal set $\sZ$ in this representation 
is given by the extreme points of $\sC$.   
\end{proof}

\section{The partially ordered set of half-spaces}

In this section we prove the existence of minimal half-spaces 
with respect to a max-plus cone $\sV$. With this aim, it is convenient 
to start with the following lemma which shows that any half-space 
$\sH$ of $\rmax^n$ can be written as 
\[
\sH=\set{x\in \rmax^n}{\oplus_{i\in I} a_i x_i\leq \oplus_{j\in J} a_j x_j} \; , 
\]
where $I$ and $J$ are disjoint subsets of $\{ 1, \ldots ,n \}$ 
and $a_k\in \R$ for all $k\in I\cup J$. Henceforth, 
all the half-spaces we consider will be written in this way. 

\begin{lemma}
Let $a,b,c,d\in \rmax$. Then, 
\[ 
\set{x\in \rmax}{a x \oplus c \leq b x \oplus d}=
\set{x\in \rmax}{a x \oplus c \leq d}
\]
if $a>b$, and 
\[ 
\set{x\in \rmax}{a x \oplus c \leq b x \oplus d}=
\set{x\in \rmax}{c \leq b x \oplus d}
\]
if $a\leq b$.
\end{lemma}

\begin{proof}
We only prove the case $a> b$ because the other one is straightforward. 

Assume that $a x \oplus c \leq b x \oplus d$. 
If $x=\zero$, necessarily $c\leq d$ 
and thus $a x \oplus c = c \leq d$. 
If $x \neq \zero$, then, 
as $a x \oplus c \geq a x > b x$ and 
$a x \oplus c \leq b x \oplus d$, 
it follows that $a x \oplus c \leq d$.  

Conversely, assume that $a x \oplus c \leq d$. Then,   
we have $a x \oplus c\leq d \leq b x \oplus d$. 
\end{proof}

Given $v\in \rmax^n$ and $\sV\subset \rmax^n$, the  
\new{supports} of $v$ and $\sV$ are respectively defined by
\[ 
\supp v:=\set{k}{v_k\neq \zero}
\enspace \mbox{ and } \enspace 
\supp \sV:=\cup_{v\in \sV}\supp v
\enspace . 
\]
We shall say that a max-plus cone $\sV\subset \rmax^n$ 
has \new{full support} if $\supp \sV=\{1,\dots ,n\}$. 
For any non-zero scalar $\lambda \in \rmax$, 
we define $\myminus{\lambda }:=-\lambda$,  
and we extend this notation to vectors of $\rmax^n$ with only finite entries, 
so that $\myminus{x}$ represents the vector with entries $-x_i$ 
for $i\in \{1,\dots ,n\}$.   
 
\begin{lemma}\label{LemmaInclu}
Let 
\[
\sH:=\set{x\in \rmax^n}{\oplus_{i\in I} a_i x_i\leq \oplus_{j\in J} a_j x_j} \;  
\]
and 
\[
\sH':=\set{x\in \rmax^n}{\oplus_{i\in I'} b_i x_i\leq \oplus_{j\in J'} b_j x_j}   
\]
be two half-spaces. 
Then, when $I\neq \emptyset$, $\sH' \subset \sH$ 
if, and only if, $I\subset I'$, $J'\subset J$ and 
$b_j \myminus{(b_i)} \leq a_j \myminus{(a_i)}$ for all $i\in I$ and $j\in J'$.
\end{lemma}

\begin{proof}

$\Rightarrow )$ Assume that $I\not \subset I'$. 
Pick any $i\in I\setminus I'$. Then, 
$\uvector^i \in \sH'$ and $\uvector^i \not \in \sH$, 
which contradicts the fact that $\sH' \subset \sH$. 
Therefore, $I\subset I'$. 

Now assume that $J'\not \subset J$. 
Pick any $j\in J'\setminus J$ and $i\in I\subset I'$, 
and define the vector $x\in \rmax^n$ by
\begin{equation}\label{DefVector}
x_k:=
\left\{
\begin{array}{cl}
\myminus{b_k} & \makebox{ if }\; k \in \{ i , j \} \; , \\
\zero & \makebox{ otherwise.} 
\end{array}
\right.
\end{equation}
Then, $x\in \sH'$ and $x \not \in \sH$, 
which is a contradiction. Therefore, $J'\subset J$. 

Finally, since the vector $x$ defined in~\eqref{DefVector} 
belongs to $\sH'$ for any 
$i\in I\subset I'$ and $j\in J'$, it follows that it also belongs to
$\sH$ and thus $a_i \myminus{(b_i)} \leq a_j \myminus{(b_j)}$. 
Therefore, $b_j \myminus{(b_i)} \leq a_j \myminus{(a_i)}$ for all 
$i\in I$ and $j\in J'$. 

$\Leftarrow )$ Since
\begin{eqnarray*}
&& x\in \sH' \implies b_i x_i \leq \oplus_{j\in J'} b_j x_j \; , \; 
\forall i\in I' \implies b_i x_i \leq \oplus_{j\in J'} b_j x_j \; , \; 
\forall i\in I  \implies \\   
&& x_i \leq \oplus_{j\in J'} b_j \myminus{(b_i)} x_j \; , \; \forall i\in I \implies
x_i \leq \oplus_{j\in J'}a_j \myminus{(a_i)} x_j \; , \; \forall i\in I \implies \\ 
&&  a_i x_i  \leq \oplus_{j\in J'} a_j x_j \; , \;  \forall i\in I \implies  
a_i x_i \leq \oplus_{j\in J} a_j x_j \; , \;  \forall i\in I \implies 
x\in \sH \; , 
\end{eqnarray*} 
it follows that $\sH' \subset \sH$.
\end{proof}

\begin{lemma}\label{LemmaInter}
Let $\sV\subset \rmax^n$ be a max-plus cone with full support 
and $\{ \sH_r \}_{r\in \N}$ 
be a decreasing sequence of half-spaces such that 
$\sV\subset \sH_r$ for all $r \in \N$. 
Then, there exists a half-space $\sH$ such that $\sH=\cap_{r\in \N} \sH_r\; .$
\end{lemma}

\begin{proof}
Assume that 
\[
\sH_r=\set{x\in \rmax^n}{\oplus_{i\in I_r} a^r_i x_i \leq \oplus_{j\in J_r} a^r_j x_j} \; ,
\]
where for all $r\in \N$, $I_r$ and $J_r$ are disjoint subsets of 
$\{ 1, \ldots ,n \}$ and $a^r_k\in \R$ for all $k\in I_r\cup J_r$. 
By Lemma~\ref{LemmaInclu} we may assume, without loss of generality, 
that there exist $I,J\subset \{1, \ldots ,n\}$  
such that $I_r=I$ and $J_r=J$ for all $r\in \N$. 
If $I=\emptyset$, we have $\sH_r=\rmax^n$ for all $r\in \N$, 
so in this case the result is obvious. We next consider the 
case $I\neq \emptyset$. Note that in this case we also have
$J\neq \emptyset$, because $\sV\subset \sH_r$ for all 
$r\in \N$ and $\supp \sV=\{ 1, \ldots ,n \}$. 

We may also assume, without loss of generality, 
that $\oplus_{j\in J} a_j^r=\unit$ for all $r\in \N$. 
Then, since $\supp \sV=\{ 1, \ldots ,n \}$ and 
$\oplus_{i\in I} a^r_i x_i \leq \oplus_{j\in J} a^r_j x_j$ 
for $r\in \N$ and $x\in \sV$, it follows that the sequence 
$\{a_i^r \}_{r\in \N}$ is bounded from above for all $i\in I$. 
Therefore, we may assume, taking sub-sequences if necessary, 
that there exists $a_i\in \rmax$ such that 
$\lim_{r\rightarrow \infty} a^r_i=a_i$ for $i\in I$.  

We claim that $a_i\neq \zero $ for all $i\in I$. 
By the contrary, assume that $a_h =\zero $ for some 
$h\in I$. Since by Lemma~\ref{LemmaInclu} we have 
$a_j^{r}(\myminus{a_h^{r})} \leq a_j^1(\myminus{a_h^1)}$ for all 
$j\in J$ and $r \in \N$, this implies that 
$\lim_{r\rightarrow \infty} a^r_j= \zero$ for all $j\in J$, 
which contradicts the fact that $J\neq \emptyset $ and 
$\oplus_{j\in J} a_j^r=\unit$ for all $r\in \N$. 
This proves our claim.  

Since $\oplus_{j\in J} a_j^r=\unit$ for all $r\in \N$, 
the sequence $\{ a_j^r \}_{r\in \N}$ is also bounded from above 
for all $j\in J$.  
Then, taking sub-sequences if necessary, 
we may assume that there exists $a_j\in \rmax$ such that 
$\lim_{r\rightarrow \infty} a^r_j=a_j$ for all $j\in J$. 
If we define $J':=\set{j\in J}{a_j \neq \zero}$ and 
\[
\sH:=\set{x\in \rmax^n}{\oplus_{i\in I} a_i x_i \leq \oplus_{j\in J'} a_j x_j} \; , 
\] 
then it follows that $\sH=\cap_{r\in \N} \sH_r$. 
Indeed, if $x\in \cap_{r\in \N} \sH_r$, we have
\[
\oplus_{i\in I} a_i x_i = 
\lim_{r\rightarrow \infty} \left( \oplus_{i\in I} a^r_i x_i \right) \leq 
\lim_{r\rightarrow \infty} \left(\oplus_{j\in J} a^r_j x_j \right) =
\oplus_{j\in J'} a_j x_j \; , 
\]
and thus $x\in \sH$. Therefore, 
$\cap_{r\in \N} \sH_r \subset \sH$. 
Conversely, since $J'\subset J$ and 
$a_j \myminus{(a_i)}\leq a_j^r \myminus{(a_i^r)}$ 
for all $i\in I$, $j\in J'$ and $r\in \N$, by Lemma~\ref{LemmaInclu} 
it follows that $\sH\subset \sH_r$ for all $r\in \N$. 
Therefore, we also have $\sH\subset \cap_{r\in \N} \sH_r$.
\end{proof}

\begin{remark}
Lemma~\ref{LemmaInter} does not hold if $\sV$ does not have full support. 
For instance, consider 
\[
\sV=\set{x\in \rmax^3}{x_2=\zero,x_1\leq x_3}
\] 
and the decreasing sequence of half-spaces
\[
\sH_r=\set{x\in \rmax^3}{x_1\oplus r x_2 \leq x_3 } \; ,  
\]
where $r\in \N$. Then, 
$\cap_{r\in \N} \sH_r=\sV$, but $\sV$ is not a half-space of $\rmax^3$.
\end{remark} 

\begin{theorem}\label{TheoExistanceMinimal}
Let $\sV\subset \rmax^n$ be a max-plus cone with full support. 
If $\sV$ is contained in the half-space $\sH$, 
then there exists a half-space $\sH'$ such that  
$\sV\subset \sH'\subset \sH$ and $\sH'$ 
is minimal for inclusion with respect to this property.
\end{theorem}

\begin{proof}
By Zorn's Lemma it suffices to show that for any chain 
$\{ \sH_\alpha \}_{\alpha\in \Delta}$ of half-spaces 
which satisfies $\sV\subset \sH_\alpha \subset \sH$ for all 
$\alpha \in \Delta$, there exists a half-space $\sH'$ such that 
$\sH'=\cap_{\alpha\in \Delta} \sH_\alpha$.   

According to Lemma~\ref{LemmaInclu}, we may assume that 
\[
\sH_\alpha=\set{x\in \rmax^n}{\oplus_{i\in I} a_i^\alpha x_i \leq \oplus_{j\in J} a_j^\alpha x_j}  
\]
for all $\alpha \in \Delta$, where $I$ and $J$ are disjoint subsets of 
$\{1,\ldots ,n\}$ and $a_k^\alpha \in \R$ for all $k\in I\cup J$. 
Again, if $I= \emptyset$ the previous assertion is trivial, 
so assume $I\neq \emptyset$. Consider any sequence 
$\{ \alpha_r \}_{r\in \N}\subset \Delta$ such that 
the sequence $\{ a_j^{\alpha_r}\myminus{(a_i^{\alpha_r})} \}_{r\in \N}$ 
is decreasing and
\begin{equation}~\label{limite}
\lim_{r\rightarrow \infty} a_j^{\alpha_r}\myminus{(a_i^{\alpha_r})}
=\inf_{\alpha \in \Delta} a_j^{\alpha}\myminus{(a_i^{\alpha})} \; ,
\end{equation}
for all $i\in I$ and $j\in J$. Then, by Lemma~\ref{LemmaInter} 
there exists a half-space $\sH'$ such that 
$\sH'=\cap_{r\in \N} \sH_{\alpha_r}$. 
Since~\eqref{limite} is satisfied, by Lemma~\ref{LemmaInclu}, for 
any $\alpha \in \Delta$ there exists $r\in \N$ such that 
$\sH_{\alpha_r}\subset \sH_\alpha $. Therefore, we have 
$\sH'=\cap_{\alpha\in \Delta} \sH_\alpha $.  
\end{proof}

As a consequence of Theorem~\ref{TheoExistanceMinimal} and 
the separation theorem for closed cones of~\cite{zimmerman77,shpiz,cgqs04}, 
it follows that any closed cone $\sV$ with full support can 
be expressed as the intersection of a family of minimal half-spaces 
with respect to $\sV$. When $\sV$ is finitely generated, 
by Theorem~\ref{MinkowskiWeylCones} 
we conclude that it is possible to select a finite number of minimal 
half-spaces with respect to $\sV$ such that their intersection is equal to 
$\sV$. However, as in the classical case, 
even in the finitely generated case, 
the number of minimal half-spaces with respect to $\sV$ need 
not be finite, as it is shown in the next section. 

\section{Characterization of minimal half-spaces}

Throughout this section, $\sV\subset \rmax^n$ will represent a fixed max-plus 
cone generated by the vectors $v^{r} \in \rmax^n$, where $r=1,\ldots , p$. 
For the sake of simplicity, in this section we shall assume that 
all the vectors we consider have only finite entries. 
We next recall basic definitions and properties related to the natural 
cell decomposition of $\rmax^n$ induced by the generators of $\sV$. 
We refer the reader to~\cite{DS04} for a complete presentation,  
but we warn that the results of~\cite{DS04} are in the setting of 
the min-plus semiring $\rmin:=(\Rmin ,\min ,+)$, 
which is however equivalent to the setting considered here.    

We define the type of a vector $x\in \rmax^n$ 
relative to the generators $v^{r}$ as the $n$-tuple 
$\type(x)=(S_1(x),\ldots , S_n(x))$ of subsets 
$S_j(x)\subset \left\{1,\ldots ,p\right\}$ 
defined as follows: 
\begin{equation} 
S_j(x):=\set{r}{ v^{r}_j \myminus{(x_j)}= \oplus_{1\leq k \leq n} v^{r}_k \myminus{(x_k)}} 
\; .
\end{equation}
Note that 
$v^{r}_j \myminus{(x_j)} < \oplus_{1\leq k \leq n} v^{r}_k \myminus{(x_k)}$ 
if $r\not \in S_j(x)$ and that any $r\in \left\{1,\ldots ,p\right\}$ 
belongs to some $S_j(x)$.  

Given a $n$-tuple $S=(S_1,\ldots , S_n)$ of subsets of 
$ \left\{1,\ldots ,p \right\} $, consider like in~\cite{DS04} 
the set $X_S$ of all the vectors whose type contains $S$, 
i.e. 
\begin{equation}
X_S:=\set{x\in \rmax^n}{S_j\subset S_j(x),\; \forall j=1,\ldots ,n } \; .
\end{equation} 
Lemma~10 of~\cite{DS04} shows that the sets $X_S$ are closed convex polyhedra 
(both in the max-plus and usual sense) which are given by  
\[
X_S=\set{x\in \rmax^n}{x_j v^r_i\leq x_iv^r_j ,\; \forall i,j\in \{1,\dots ,n\} 
\makebox{ with }r\in S_j } \; .
\] 
The natural cell decomposition of $\rmax^n$ induced 
by the generators of $\sV$ is the collection of convex polyhedra 
$X_S$, where $S$ ranges over all the possible types. 
This cell decomposition has in particular the property that $\sV$ 
is the union of its bounded cells, where a cell is said to be bounded 
if it is bounded in the $(n-1)$-dimensional \new{max-plus or tropical 
projective space} $\mathbb{R}^n / (1,\ldots ,1)\R$   
(see~\cite{DS04} for details). 

Given a cell $X_S$, if we define the undirected graph $G_S$ with set of nodes 
$\{1,\ldots ,n\}$ and an edge connecting the nodes $i$ and $j$ 
if and only if $S_i\cap S_j\neq \emptyset$, 
then by Proposition~17 of~\cite{DS04} the dimension of $X_S$ 
is given by the number of connected components of $G_S$ 
(in~\cite{DS04} the dimension of $X_S$ is one less the one 
considered here because it refers to the projective space). 
Any non-zero vector in a cell of dimension one, which is therefore 
of the form $\set{\lambda x}{\lambda \in \rmax}$ for some $x\in \R^n$, 
is called a \new{vertex} of the natural cell decomposition. 

\begin{example}
Consider the max-plus cone $\sV \subset \rmax^3$ generated by the 
vectors: $v^1=(1,2,3)^T$, $v^2=(2,4,6)^T$ and $v^3=(3,6,9)^T$. 
This cone is represented on the left hand side of Figure~\ref{FigureTypes} 
by the bounded dark gray region together with the two line segments 
joining the points $v^1$ and $v^3$ to it. On the same figure 
we show the type of a vector, for each cell $X_S$ contained in $\sV$. 
For instance, the type of the vector $a=(0,1,3)^T$ is 
$S=\type(a)=(\{1\},\{1,2\},\{2,3\})$. Then, since the graph $G_{S}$ 
has only one connected component, $a$ is a vertex. 
If we take $b=(0,1,2.5)^T$, then $S=\type(b)=(\{1\},\{1\},\{2,3\})$ 
so that in this case the cell $X_{S}$ is two dimensional. 
In Figure~\ref{FigureTypes} this cell is represented by the 
line segment which connects the points $v^1$ and $a$.   
The natural cell decomposition of $\rmax^3$ induced by the generators of 
$\sV$ has six vertices, fifteen two dimensional cells (six of them bounded) 
and ten three dimensional cells (only one of them bounded, 
which is represented by the bounded dark gray region labeled 
by the type $S=(\{1\},\{2\},\{3\})$ 
on the left hand side of Figure~\ref{FigureTypes}). 
  
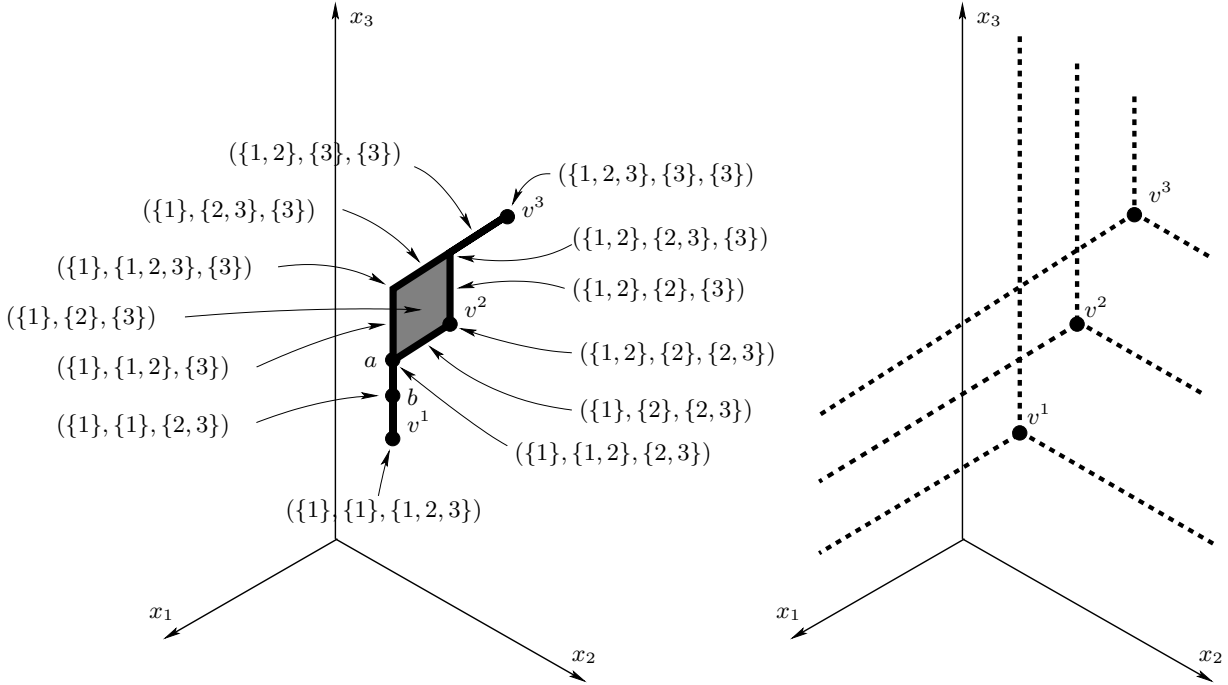
\begin{figure}
\begin{center}
\begin{picture}(0,0)%
\includegraphics{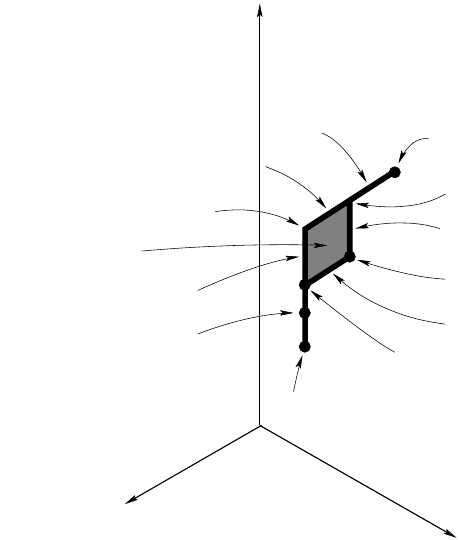}%
\end{picture}%
\setlength{\unitlength}{2368sp}%
\begingroup\makeatletter\ifx\SetFigFont\undefined%
\gdef\SetFigFont#1#2#3#4#5{%
  \reset@font\fontsize{#1}{#2pt}%
  \fontfamily{#3}\fontseries{#4}\fontshape{#5}%
  \selectfont}%
\fi\endgroup%
\begin{picture}(6112,7169)(2536,-5483)
\put(4051,-4786){\makebox(0,0)[lb]{\smash{{\SetFigFont{9}{10.8}{\rmdefault}{\mddefault}{\updefault}{\color[rgb]{0,0,0}$x_1$}%
}}}}
\put(6751,-2836){\makebox(0,0)[lb]{\smash{{\SetFigFont{9}{10.8}{\rmdefault}{\mddefault}{\updefault}{\color[rgb]{0,0,0}$v^1$}%
}}}}
\put(7351,-1636){\makebox(0,0)[lb]{\smash{{\SetFigFont{9}{10.8}{\rmdefault}{\mddefault}{\updefault}{\color[rgb]{0,0,0}$v^2$}%
}}}}
\put(5476,-3736){\makebox(0,0)[lb]{\smash{{\SetFigFont{9}{10.8}{\rmdefault}{\mddefault}{\updefault}{\color[rgb]{0,0,0}$(\{1\},\{1\},\{1,2,3\})$}%
}}}}
\put(8326,-211){\makebox(0,0)[lb]{\smash{{\SetFigFont{9}{10.8}{\rmdefault}{\mddefault}{\updefault}{\color[rgb]{0,0,0}$(\{1,2,3\},\{3\},\{3\})$}%
}}}}
\put(8476,-886){\makebox(0,0)[lb]{\smash{{\SetFigFont{9}{10.8}{\rmdefault}{\mddefault}{\updefault}{\color[rgb]{0,0,0}$(\{1,2\},\{2,3\},\{3\})$}%
}}}}
\put(8476,-1411){\makebox(0,0)[lb]{\smash{{\SetFigFont{9}{10.8}{\rmdefault}{\mddefault}{\updefault}{\color[rgb]{0,0,0}$(\{1,2\},\{2\},\{3\})$}%
}}}}
\put(6151,1439){\makebox(0,0)[lb]{\smash{{\SetFigFont{9}{10.8}{\rmdefault}{\mddefault}{\updefault}{\color[rgb]{0,0,0}$x_3$}%
}}}}
\put(3976,-586){\makebox(0,0)[lb]{\smash{{\SetFigFont{9}{10.8}{\rmdefault}{\mddefault}{\updefault}{\color[rgb]{0,0,0}$(\{1\},\{2,3\},\{3\})$}%
}}}}
\put(3076,-2836){\makebox(0,0)[lb]{\smash{{\SetFigFont{9}{10.8}{\rmdefault}{\mddefault}{\updefault}{\color[rgb]{0,0,0}$(\{1\},\{1\},\{2,3\})$}%
}}}}
\put(7876,-3136){\makebox(0,0)[lb]{\smash{{\SetFigFont{9}{10.8}{\rmdefault}{\mddefault}{\updefault}{\color[rgb]{0,0,0}$(\{1\},\{1,2\},\{2,3\})$}%
}}}}
\put(8551,-2686){\makebox(0,0)[lb]{\smash{{\SetFigFont{9}{10.8}{\rmdefault}{\mddefault}{\updefault}{\color[rgb]{0,0,0}$(\{1\},\{2\},\{2,3\})$}%
}}}}
\put(8551,-2086){\makebox(0,0)[lb]{\smash{{\SetFigFont{9}{10.8}{\rmdefault}{\mddefault}{\updefault}{\color[rgb]{0,0,0}$(\{1,2\},\{2\},\{2,3\})$}%
}}}}
\put(7951,-586){\makebox(0,0)[lb]{\smash{{\SetFigFont{9}{10.8}{\rmdefault}{\mddefault}{\updefault}{\color[rgb]{0,0,0}$v^3$}%
}}}}
\put(3076,-2236){\makebox(0,0)[lb]{\smash{{\SetFigFont{9}{10.8}{\rmdefault}{\mddefault}{\updefault}{\color[rgb]{0,0,0}$(\{1\},\{1,2\},\{3\})$}%
}}}}
\put(3076,-1186){\makebox(0,0)[lb]{\smash{{\SetFigFont{9}{10.8}{\rmdefault}{\mddefault}{\updefault}{\color[rgb]{0,0,0}$(\{1\},\{1,2,3\},\{3\})$}%
}}}}
\put(4876, 14){\makebox(0,0)[lb]{\smash{{\SetFigFont{9}{10.8}{\rmdefault}{\mddefault}{\updefault}{\color[rgb]{0,0,0}$(\{1,2\},\{3\},\{3\})$}%
}}}}
\put(2551,-1711){\makebox(0,0)[lb]{\smash{{\SetFigFont{9}{10.8}{\rmdefault}{\mddefault}{\updefault}{\color[rgb]{0,0,0}$(\{1\},\{2\},\{3\})$}%
}}}}
\put(6301,-2161){\makebox(0,0)[lb]{\smash{{\SetFigFont{9}{10.8}{\rmdefault}{\mddefault}{\updefault}{\color[rgb]{0,0,0}$a$}%
}}}}
\put(6751,-2557){\makebox(0,0)[lb]{\smash{{\SetFigFont{9}{10.8}{\rmdefault}{\mddefault}{\updefault}{\color[rgb]{0,0,0}$b$}%
}}}}
\put(8476,-5236){\makebox(0,0)[lb]{\smash{{\SetFigFont{9}{10.8}{\rmdefault}{\mddefault}{\updefault}{\color[rgb]{0,0,0}$x_2$}%
}}}}
\end{picture}%
\enspace\; \;\;\;\;\;\;\;\;\;\;\;\;\;\;\;\;\;\;
\begin{picture}(0,0)%
\includegraphics{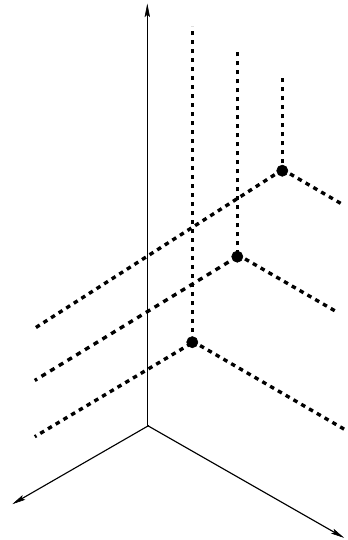}%
\end{picture}%
\setlength{\unitlength}{2368sp}%
\begingroup\makeatletter\ifx\SetFigFont\undefined%
\gdef\SetFigFont#1#2#3#4#5{%
  \reset@font\fontsize{#1}{#2pt}%
  \fontfamily{#3}\fontseries{#4}\fontshape{#5}%
  \selectfont}%
\fi\endgroup%
\begin{picture}(4634,7169)(4036,-5483)
\put(7951,-436){\makebox(0,0)[lb]{\smash{{\SetFigFont{9}{10.8}{\rmdefault}{\mddefault}{\updefault}{\color[rgb]{0,0,0}$v^3$}%
}}}}
\put(7276,-1636){\makebox(0,0)[lb]{\smash{{\SetFigFont{9}{10.8}{\rmdefault}{\mddefault}{\updefault}{\color[rgb]{0,0,0}$v^2$}%
}}}}
\put(6676,-2761){\makebox(0,0)[lb]{\smash{{\SetFigFont{9}{10.8}{\rmdefault}{\mddefault}{\updefault}{\color[rgb]{0,0,0}$v^1$}%
}}}}
\put(4051,-4786){\makebox(0,0)[lb]{\smash{{\SetFigFont{9}{10.8}{\rmdefault}{\mddefault}{\updefault}{\color[rgb]{0,0,0}$x_1$}%
}}}}
\put(6151,1439){\makebox(0,0)[lb]{\smash{{\SetFigFont{9}{10.8}{\rmdefault}{\mddefault}{\updefault}{\color[rgb]{0,0,0}$x_3$}%
}}}}
\put(8476,-5236){\makebox(0,0)[lb]{\smash{{\SetFigFont{9}{10.8}{\rmdefault}{\mddefault}{\updefault}{\color[rgb]{0,0,0}$x_2$}%
}}}}
\end{picture}%
\end{center}
\caption{Illustration of the combinatorial types (left) and of the natural cell decomposition of $\rmax^3$ induced 
by the generators of a max-plus cone (right), as defined by Develin and Sturmfels~\cite{DS04}.}
\label{FigureTypes}
\end{figure}

A simple geometric construction of the natural cell decomposition of 
$\rmax^n$ induced by the generators of $\sV$ can be obtained if we consider the 
min-plus hyperplanes whose apices are the generators of $\sV$. 
Given $a\in \rmax^n$, the \new{min-plus hyperplane} with apex 
$\myminus{a}$ is the set of vectors $x\in \rmax^n$ such that the minimum  
$\min_{1\leq i\leq n} a_i +x_i$ is attained at least twice 
(we refer the reader to~\cite{joswig04} for details on  
hyperplanes and their relation with half-spaces).  
By Proposition~16 of~\cite{DS04}, the cell decomposition induced 
by the generators of $\sV$ is the common refinement of the fans defined by the 
the $p$ min-plus hyperplanes whose apices are the vectors $v^r$, 
for $r=1,\ldots ,p$. In the case of our example, 
these min-plus hyperplanes are 
represented on the right hand side of Figure~\ref{FigureTypes}, 
where it can be seen that $\sV$ is the union of the bounded cells.  
\end{example}

Assume that 
\[
\sH=\set{x\in \rmax^n}{\oplus_{i\in I} a_i x_i \leq \oplus_{j\in J} a_j x_j} \;  
\]
is a minimal half-space with respect to $\sV$. Then, we necessarily have  
$I\cup J=\{1,\dots ,n\}$. Indeed, if $h\not \in I\cup J$, 
defining $a_h=\min_{1\leq r \leq p}\{\oplus_{j\in J} a_j v_j^r \myminus{(v_h^r)}\}$, 
it follows that the half-space 
\[
\sH'=\set{x\in \rmax^n}{a_h x_h \oplus (\oplus_{i\in I} a_i x_i )\leq \oplus_{j\in J} a_j x_j} \;  
\]
contains $\sV$ because it contains its generators, 
which by Lemma~\ref{LemmaInclu} contradicts the minimality of $\sH$. 
For this reason, 
in this section we shall assume that $I\cup J=\{1,\dots ,n\}$  
since we are interested in studying minimal half-spaces, 
and like in~\cite{joswig04} we shall call the vector 
$\myminus{a}\in \sH \subset \rmax^n$ the \new{apex} of $\sH$. 

The following lemma gives a necessary and sufficient condition for 
$\sV$ to be contained in a half-space in terms of the type of its apex. 

\begin{lemma}\label{Lema1}
The max-plus cone $\sV$ is contained in the half-space 
\[
\sH=\set{x\in \rmax^n}
{\oplus_{i\in I} a_i x_i \leq \oplus_{j\in J} a_j x_j} 
\]
with apex $\myminus{a}$ if, and only if, 
$\cup_{j\in J} S_j(\myminus{a})=\left\{ 1,\ldots ,p\right\}$. 
\end{lemma}

\begin{proof}
Assume that $\cup_{j\in J} S_j(\myminus{a})\neq \left\{ 1,\ldots ,p\right\}$. 
Then, there exists $r\in \left\{1,\ldots ,p\right\} $ such that 
$r\not \in S_j(\myminus{a})$ for all $j\in J$, and so  
$a_j v^{r}_j < \oplus_{1\leq k \leq n} a_k v^{r}_k$ for all $j\in J$. 
Therefore, we have  
\[ 
\oplus_{j\in J} a_j v^{r}_j < \oplus_{1\leq k \leq n} a_k v^{r}_k= 
\oplus_{i\in I} a_i v^{r}_i \; ,
\]  
which means that $v^{r}$ does not belong to $\sH$ and so 
$\sV$ is not contained in $\sH$. 
This shows the ``only if'' part of the lemma. 

Now assume that $\cup_{j\in J} S_j(\myminus{a})=\left\{ 1,\ldots ,p\right\}$. 
Then, for each $r\in \left\{1, \ldots ,p \right\}$ there exists $j\in J$ 
such that $r\in S_j(\myminus{a})$, which means that 
$a_j v^{r}_j=\oplus_{1\leq k \leq n} a_k v^{r}_k$. 
Therefore, we have  
\[ 
\oplus_{i\in I} a_i v^{r}_i \leq \oplus_{1\leq k \leq n} a_k v^{r}_k= 
\oplus_{j\in J} a_j v^{r}_j \; . 
\] 
Since this holds for all $r\in \left\{1, \ldots ,p \right\}$, 
it follows that $\sH$ contains the generators of $\sV$, 
and thus $\sV$ is contained in $\sH$.  
This proves the ``if'' part of the lemma.
\end{proof}

Now we can characterize the minimality of a half-space 
with respect to $\sV$ in terms of the type of its apex.

\begin{theorem}\label{TheoCharacMin}
The half-space  
\[
\sH=\set{x\in \rmax^n}
{\oplus_{i\in I} a_i x_i \leq \oplus_{j\in J} a_j x_j} 
\]
with apex $\myminus{a}$ is minimal with respect to the max-plus cone $\sV$ if, 
and only if, the following conditions are satisfied: 
\begin{enumerate}[(i)]
\item\label{item:C1} For each $i\in I$ there exists $j\in J$ such that 
$S_i(\myminus{a})\cap S_j(\myminus{a}) \neq \emptyset $,
\item\label{item:C2} For each $j\in J$ there exists $i\in I$ such that 
$S_i(\myminus{a})\cap S_j(\myminus{a}) \not \subset 
\cup_{k\in J\setminus\left\{ j \right\} } S_k(\myminus{a})$,  
\item\label{item:C3} $\cup_{j\in J} S_j(\myminus{a})= 
\left\{ 1,\ldots ,p\right\}$.
\end{enumerate} 
\end{theorem}

\begin{proof}
If $\sH$ is minimal with respect to $\sV$, for each $i\in I$ there exists 
$r\in \left\{1, \ldots ,p \right\}$ such that 
$a_i v^r_i =\oplus_{j\in J} a_j v^r_j$ because otherwise 
there would exist $\delta >0 $ such that 
$\delta a_i v^r_i \oplus \left( \oplus_{h\in I\setminus\left\{ i\right\} } 
a_h v^r_h \right) \leq \oplus_{j\in J} a_j v^r_j$ for all 
$r\in \left\{1, \ldots ,p \right\}$, 
contradicting by Lemma~\ref{LemmaInclu} the minimality of $\sH$. 
Therefore, for each $i\in I$ there exist 
$r\in \left\{1, \ldots ,p \right\}$ and $j\in J$ 
such that $a_i v^r_i=a_j v^r_j\geq a_k v^r_k $ for all $k$, 
which implies that $r\in S_i(\myminus{a})\cap S_j(\myminus{a})$ and so 
$S_i(\myminus{a})\cap S_j(\myminus{a})\neq \emptyset $. 

Analogously, if $\sH$ is minimal with respect to $\sV$, 
for each $j\in J$ there exists 
$r\in \left\{1, \ldots ,p \right\}$ such that 
$\oplus_{i\in I} a_i v^r_i = a_j v^r_j >  
\oplus_{k\in J\setminus\left\{ j \right\} } a_k v^r_k $. Otherwise, 
there would exist $\delta <0 $ such that 
$\oplus_{i\in I} a_i v^r_i \leq \delta a_j v^r_j \oplus 
\left( \oplus_{k\in J\setminus\left\{ j\right\} } a_k v^r_k  \right) $ 
for all $r\in \left\{1, \ldots ,p \right\}$, 
which by Lemma~\ref{LemmaInclu} contradicts the minimality of $\sH$.  
Therefore, for each $j\in J$  
there exist $r\in \left\{1, \ldots ,n \right\}$ and $i\in I$  
such that $a_i v^r_i = a_j v^r_j\geq a_k v^r_k $ for all $k$, 
where the inequality is strict for $k\in J\setminus\left\{ j \right\} $, 
which implies that $r\in S_i(\myminus{a})\cap S_j(\myminus{a})$ but 
$r\not \in \cup_{k\in J\setminus\left\{ j \right\} }S_k(\myminus{a})$. 

Finally, since any minimal half-space with respect to $\sV$ contains in particular $\sV$, 
it follows that $\cup_{j\in J} S_j(\myminus{a})=\left\{ 1,\ldots ,p\right\}$ 
by Lemma~\ref{Lema1}. 
This completes the proof of the ``only if'' part of the theorem. 

Now assume that the three conditions of the theorem are satisfied. 
By Lemma~\ref{Lema1}, Condition~\eqref{item:C3} implies that $\sV$ 
is contained in $\sH$. Let 
\[
\sH' = \set{x\in \rmax^n}
{\oplus_{i\in I'} b_i x_i  \leq \oplus_{j\in J'} b_j x_j } 
\]
be a minimal half-space with respect to $\sV$ contained in $\sH$, 
which we know exists by Theorem~\ref{TheoExistanceMinimal}. 
Then, since $\sH' \subset \sH $, by Lemma~\ref{LemmaInclu} 
we have $I\subset I'$, $J'\subset J$ and 
$b_j \myminus{(b_i)} \leq a_j \myminus{(a_i)}$ for all $i\in I$ and $j\in J'$. 

We first show that $I=I'$, and thus $J=J'$. 
By the contrary, assume that $I\neq I'$ and let $h\in I'\setminus I\subset J$. 
Then, by Condition~\eqref{item:C2}, there exist $l\in I$ and 
$r\in \left\{1, \ldots ,p \right\}$ such that 
$a_l v_l^r = a_h v_h^r  > a_k v_k^r$ for all $k\in J\setminus \{h\}$. 
Therefore, 
since $b_j \myminus{(b_i)} \leq a_j \myminus{(a_i)}$ 
for all $i\in I$ and $j\in J'$, we have 
\[ 
v_l^r > \oplus_{k\in J\setminus \{h\} } a_k \myminus{(a_l)} v_k^r \geq 
\oplus_{k\in J'}  a_k \myminus{(a_l)} v_k^r \geq \oplus_{k\in J'} b_k \myminus{(b_l)} v_k^r  
\]
and so 
\[ 
\oplus_{i\in I'} b_i v_i^r \geq \oplus_{i\in I} b_i v_i^r \geq 
b_l v_l^r > \oplus_{k\in J'} b_k v_k^r  \; ,
\]
contradicting the fact that $v^r \in \sV \subset \sH'$. Therefore, 
we conclude that $I=I'$ and $J=J'$. 

Note that by Conditions~\eqref{item:C1} and~\eqref{item:C2}, 
it follows that $S_k(\myminus{a})\neq \emptyset$ for all $k$. 
This implies, by the covering theorem of Vorobyev~\cite[Th.~2.6]{vorobyev67}
and Zimmermann~~\cite[Ch.~3]{Zimmermann.K}
(see~\cite{agk04} for a complete recent discussion, including generalizations;
see also Corollary~12 and Theorem~ 15 of~\cite{DS04}), 
that the apex $\myminus{a}$ of $\sH$ belongs to $\sV$. 
Therefore, since $\sV\subset \sH'$, we have 
$\oplus_{i\in I} b_i \myminus{(a_i)} \leq \oplus_{j\in J} b_j \myminus{(a_j)}$. 
Without loss of generality, we may assume that 
$\oplus_{i\in I} b_i \myminus{(a_i)} \leq \oplus_{j\in J} b_j \myminus{(a_j)}=\unit $. 
Then, since $b_j \myminus{(a_j)}\leq b_i \myminus{(a_i)}$ for all $i\in I$ and $j\in J$, 
we must have $b_i \myminus{(a_i)}=\unit $, i.e.\ $a_i=b_i$, for all $i\in I$. 

Now assume that $a\neq b$. Then, 
there exists $j\in J$ such that $b_j < a_j$ 
(note that $b_k \leq a_k $ for all $k\in J$ 
because $\oplus_{k\in J} b_k \myminus{(a_k)}=\unit $), 
and by Condition~\eqref{item:C2} there exist $i\in I$ and 
$r\in \left\{1, \ldots ,p \right\}$ such that 
\[ 
a_i v_i^r=a_j v_j^r > a_k v_k^r 
\]
for all $k\in J\setminus \{j\}$. Therefore, it follows that 
\[ 
b_i v_i^r = a_i v_i^r = a_j v_j^r > b_j v_j^r 
\] 
and 
\[ 
b_i v_i^r= a_i v_i^r  > a_k v_k^r\geq b_k v_k^r
\] 
for all $k\in J\setminus \{j\}$, implying that 
\[ 
\oplus_{h\in I} b_h v_h^r \geq b_i v_i^r > \oplus_{k\in J} b_k v_k^r  
\; , 
\] 
which contradicts the fact that $v^r \in \sV \subset \sH'$. 
In consequence, we conclude that $a=b$, and so $\sH=\sH'$, 
showing that $\sH$ is a minimal half-space with respect to $\sV$. 
\end{proof}
  
Note that the theorem above tells us that the property of being 
minimal with respect to $\sV$ depends on the type of the apex of a half-space. 
More precisely, if $\myminus{a}\in \R^n$ is the apex of a minimal 
half-space with respect to $\sV$ and $\myminus{b}$ is in the relative interior of $X_S$, 
where $S=\type(\myminus{a})$, 
then $\myminus{b}$ is also the apex of a minimal half-space with respect to $\sV$. 
Observe also that, 
as it was shown in the proof of Theorem~\ref{TheoCharacMin}, 
the conditions in this theorem imply that the apex of a minimal half-space 
with respect to $\sV$ must belong to $\sV$. 
However, these conditions do not imply that 
the apex of a minimal half-space with respect to $\sV$ should be a vertex of the 
natural cell decomposition of $\rmax^n$ induced by the generators of $\sV$. 
In other words, if $\myminus{a}$ is the apex of a minimal half-space with respect to $\sV$ 
and $S=\type(\myminus{a})$, 
then $G_S$ need not have only one connected component. 
Indeed, this is not the case, except when $n=3$. 

\begin{example}\label{Ejemplo1} 
Consider the max-plus cone $\sV \subset \rmax^4$ generated by the 
following vectors: $v^r=(1r,2r,3r,4r)^T$ for $r=1,\ldots ,4$, 
where the product is in the usual algebra.  
Note that these vectors are in general position, 
as defined in~\cite{DS04}. Indeed, this kind of cones were already 
studied in~\cite{blockyu06,AGK09} and can be seen as 
the max-plus analogues of the cyclic polytopes. 

If we take $a=(8,6,3.5,(-0.5))^T$, then 
$S=\type(\myminus{a})=(\{1,2\},\{2\},\{3,4\},\{4\})$, 
so $\myminus{a}$ is not a vertex. 
However, since the conditions of Theorem~\ref{TheoCharacMin} are satisfied for 
$I=\{2,4\}$ and $J=\{1,3\}$, it follows that 
\begin{equation}\label{MinHalfExam}
\sH = \set{x\in \rmax^4}
{6 x_2 \oplus (-0.5) x_4\leq 8 x_1 \oplus 3.5 x_3 } \;  ,
\end{equation}
or
\[
\sH = \set{x\in (\mathbb{R}\cup\{-\infty\})^4}
{\max(6 +x_2 , -0.5+ x_4)\leq \max(8+ x_1, 3.5+ x_3) } 
\]
with the usual notation,
is a minimal half-spaces with respect to $\sV$. 
Indeed, since 
\[
X_S=\set{x\in \rmax^4}{x_2=2x_1\; ,\; x_4=4x_3 \; , \; 4x_1\leq x_3 \leq 5x_1} \; ,
\] 
any half-space of the form 
\begin{align}
\set{x\in \rmax^4}
{6 x_2 \oplus \delta x_4\leq 8 x_1 \oplus \delta 4 x_3 } \; , 
\label{e-general}
\end{align}
where $-1<\delta <0$, is minimal with respect to $\sV$ because its apex 
belongs to the relative interior of $X_S$. 
Moreover, this also shows that, 
even if we assume that the generators of a max-plus cone are in general position, 
the number of minimal half-spaces need not be finite. 

This is illustrated in Figure~\ref{fig-minimal}, which shows the max-plus
cone $\sV$ (rightmost picture, in blue) 
together with two minimal half-spaces containing it
corresponding to the choice of $\delta=-0.33$ (leftmost picture) and $\delta=-0.67$ (middle picture). The apex of each of these half-spaces 
belongs to the max-plus segment joining the vectors $v^2$ and $v^4$.
The existence
of an infinite family of minimal half-spaces can be seen
on the picture:
when the apex of the half-space slides along the middle part of this max-plus
segment, the intersection of the boundary of the half-space (in green) 
with the max-plus cone yields an infinite family of sets,
two instances of which are represented.
The pictures of the max-plus polytopes were generated with {\sc Polymake}~\cite{polymake}, the latest version of which contains an extension dealing with tropical polytopes~\cite{joswig-2008}.
We plotted them with {\sc Javaview}. The bounded parts of the half-spaces
and their intersections with the max-plus cone
were computed in {\sc Scicoslab} using the {\sc Maxplus toolbox}~\cite{toolbox}. A vector $x=(x_1,\ldots,x_4)$ in $\mathbb{R}^4$
is represented by the vector $(y_1,y_2,y_3)\in \mathbb{R}^3$
with $y_i=x_{i+1}-x_1$. The plane containing the $y_1,y_2$ axes
is represented by a grid (the $y_3$ axis is a vertical
line orthogonal to this grid, not represented).
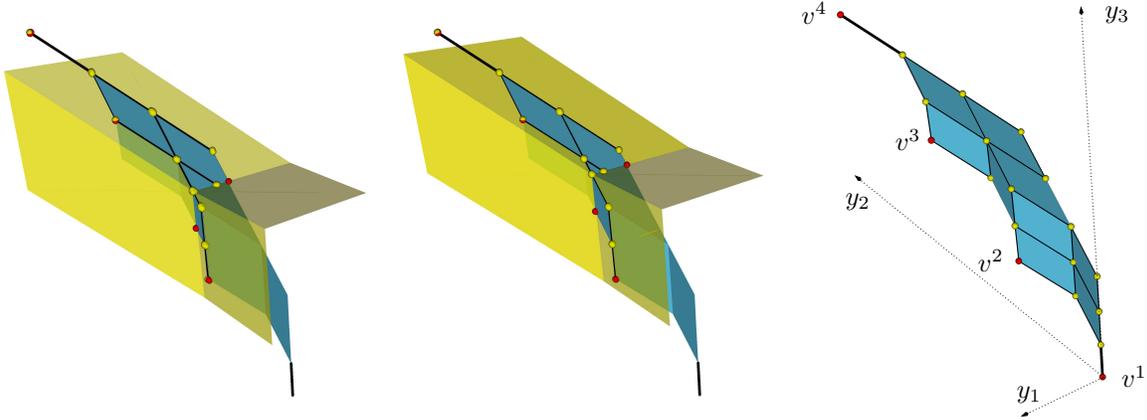
\begin{figure}
\begin{minipage}{0.32\textwidth}
\begin{center}
\includegraphics[scale=0.45]{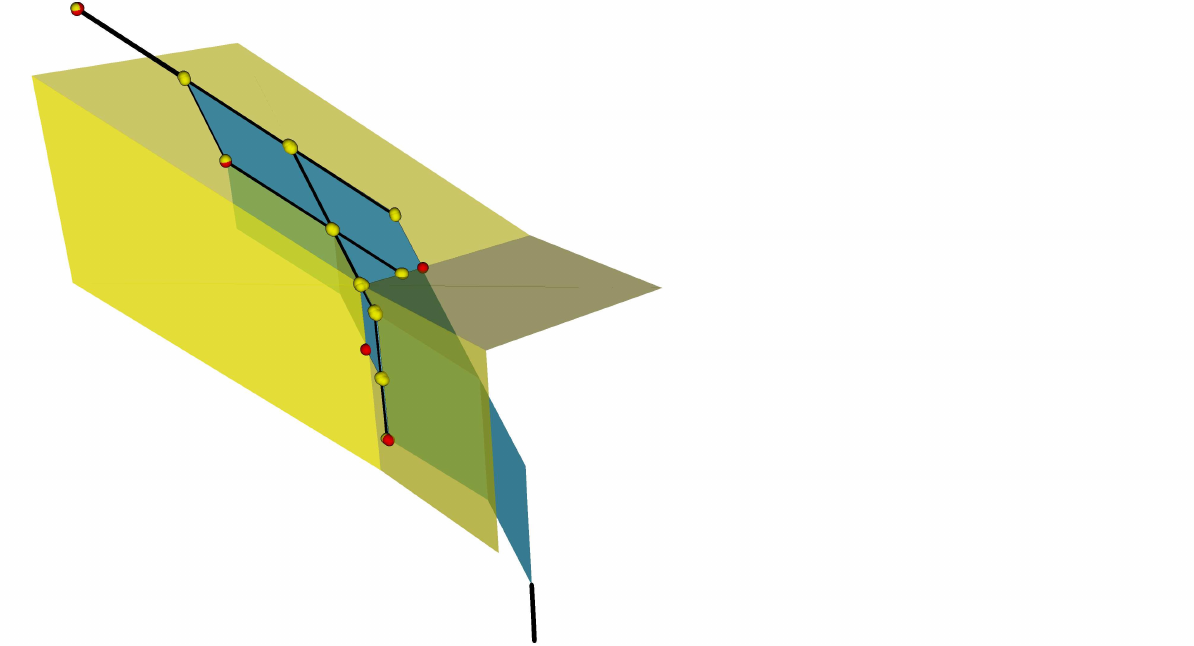}
\end{center}
\end{minipage}
\begin{minipage}{0.32\textwidth}
\begin{center}
\includegraphics[scale=0.45]{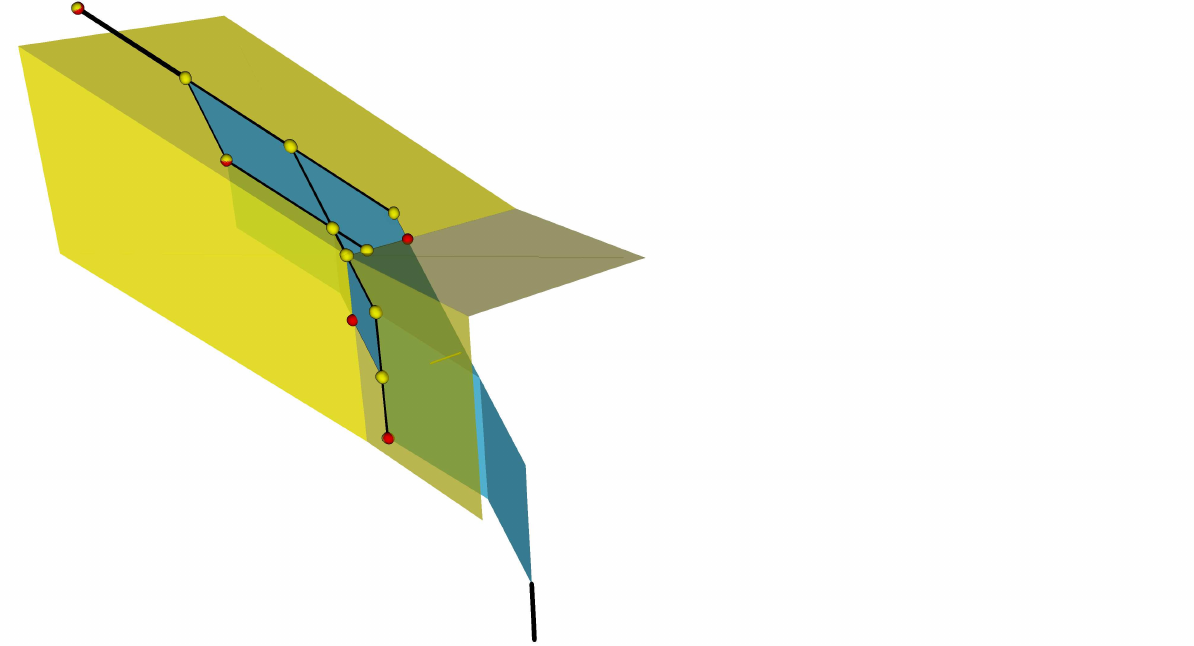}
\end{center}
\end{minipage}
\begin{minipage}{0.32\textwidth}
\begin{center}
\includegraphics[scale=0.45]{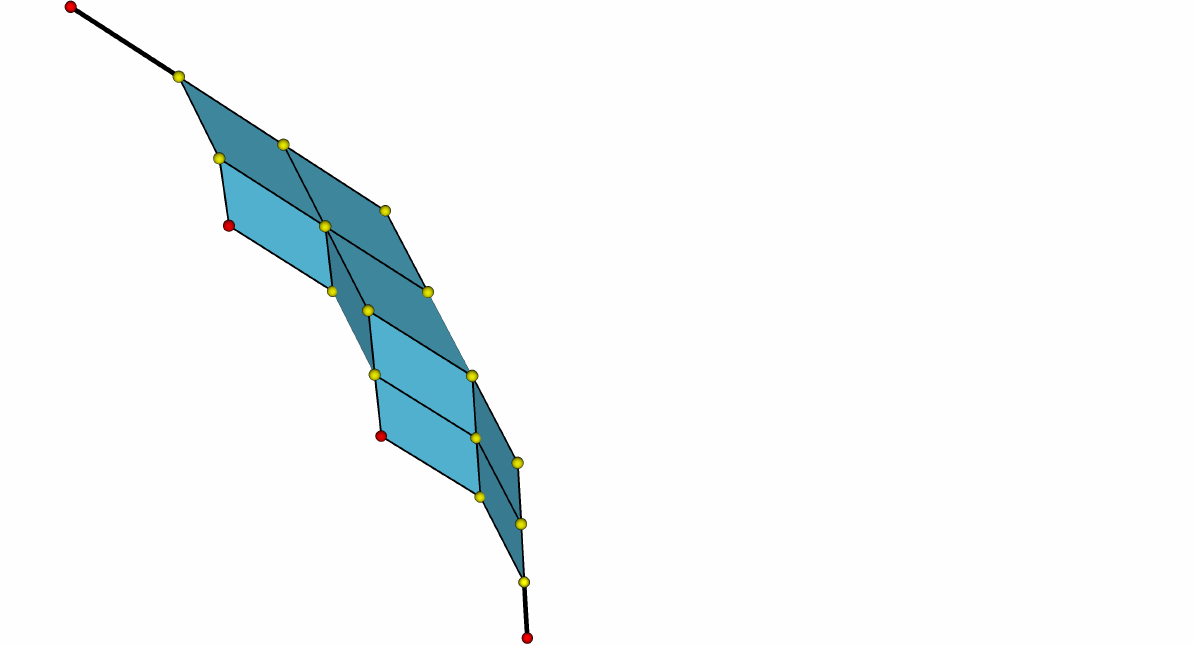}
\setlength{\unitlength}{1367sp}%
\begingroup\makeatletter\ifx\SetFigFont\undefined%
\gdef\SetFigFont#1#2#3#4#5{%
  \reset@font\fontsize{#1}{#2pt}%
  \fontfamily{#3}\fontseries{#4}\fontshape{#5}%
  \selectfont}%
\fi\endgroup%
\begin{picture}(6000,-1000)(1081,-7321)
\put(500,-200){\makebox(0,0)[lb]{\smash{{\SetFigFont{10}{13.2}{\rmdefault}{\mddefault}{\updefault}{\color[rgb]{0,0,0}$v^4$}%
}}}}
\put(6300,-6800){\makebox(0,0)[lb]{\smash{{\SetFigFont{10}{13.2}{\rmdefault}{\mddefault}{\updefault}{\color[rgb]{0,0,0}$v^1$}%
}}}}
\put(3700,-4700){\makebox(0,0)[lb]{\smash{{\SetFigFont{10}{13.2}{\rmdefault}{\mddefault}{\updefault}{\color[rgb]{0,0,0}$v^2$}%
}}}}
\put(1300,-3500){\makebox(0,0)[lb]{\smash{{\SetFigFont{10}{13.2}{\rmdefault}{\mddefault}{\updefault}{\color[rgb]{0,0,0}$y_2$}%
}}}}
\put(2200,-2506){\makebox(0,0)[lb]{\smash{{\SetFigFont{10}{13.2}{\rmdefault}{\mddefault}{\updefault}{\color[rgb]{0,0,0}$v^3$}%
}}}}
\put(6000,-100){\makebox(0,0)[lb]{\smash{{\SetFigFont{10}{13.2}{\rmdefault}{\mddefault}{\updefault}{\color[rgb]{0,0,0}$y_3$}%
}}}}
\put(4400,-7000){\makebox(0,0)[lb]{\smash{{\SetFigFont{10}{13.2}{\rmdefault}{\mddefault}{\updefault}{\color[rgb]{0,0,0}$y_1$}%
}}}}
\put(1450,-7350){\includegraphics{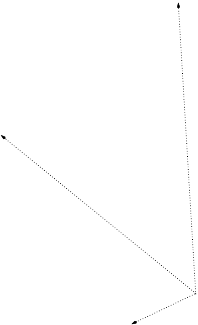}}
\end{picture}%
\end{center}
\end{minipage}
\caption{The counter-example: a max-plus polyhedral cone (right) and its intersections with two members of an infinite family of minimal half-spaces containing it (left and middle). Note that the apex of each of these minimal half-spaces is not a vertex of the 
natural cell decomposition of $\rmax^4$ induced by the generators of the cone.}\label{fig-minimal}
\end{figure}
\end{example}

\begin{remark}
As in the classical theory of convex cones, 
we could define a face of a max-plus cone $\sV\subset \rmax^n$ as its intersection 
with the closure of the complement of a minimal half-space with respect to $\sV$, 
which is also a half-space. However, unlike the classical case, 
the extreme vectors of a face of a max-plus cone defined in this way 
need not be extreme vectors of the cone, 
even in the finitely generated case. 
To see this, consider the cone $\sV \subset \rmax^4$ defined in 
Example~\ref{Ejemplo1} above and the minimal half-space 
with respect to $\sV$ given by~\eqref{MinHalfExam}. 
Then, it can be checked that the face defined by this minimal half-space 
has extreme vectors which are not extreme vectors of $\sV$. 
Two similar faces are represented in Figure~\ref{fig-minimal}. 
\end{remark}

When $n=3$, the conditions in Theorem~\ref{TheoCharacMin} imply 
that the apex of a minimal half-space with respect to $\sV$ must be a vertex 
of the natural cell decomposition of $\rmax^3$ induced by the 
generators of $\sV$. This means that $\sV$ can be expressed as a 
finite intersection of half-spaces whose apices are vertices. 
We next show that this property is also valid in higher dimensions. 
With this aim, we shall need the following 
immediate consequence of Proposition~19 of~\cite{DS04}. 

\begin{lemma}\label{LemmaVertexHull}
Let $X_S$ be a bounded cell of the natural cell decomposition 
of $\rmax^n$ induced by the generators of $\sV$. Then,   
$x\in X_S$ if and only if it can be expressed as 
$x=\min_{1\leq s\leq m} \lambda_s a^s$ 
for some scalars $\lambda_s\in \R$, where $a^s$ for 
$s\in \{1,\ldots ,m\}$ are vertices 
of the natural cell decomposition which belong to $X_S$ 
(in other words, $X_S$ is the min-plus cone generated by its vertices). 
\end{lemma}  

As a consequence, we have the following separation theorem 
in the special case of finitely generated max-plus cones 
whose generators have only finite entries. 

\begin{proposition}\label{PropApex}
Assume that $y\in \rmax^n$ does not belong to the max-plus cone $\sV$. Then, 
there exists a half-space containing $\sV$ but not $y$ whose apex 
is a vertex of the natural cell decomposition of $\rmax^n$ induced 
by the generators of $\sV$. 
\end{proposition}

\begin{proof}
By the separation theorem for closed cones of~\cite{cgqs04}, 
there exists a half-space 
\[ 
\sH = \set{x\in \rmax^n}
{\oplus_{i\in I} b_i x_i \leq \oplus_{j\in J} b_j x_j } 
\] 
containing $\sV$ but not $y$ whose apex $\myminus{b}$ belongs to $\sV$. 
To be more precise, in~\cite{cgqs04} it is shown that we can 
take $\myminus{b}=\max \set{x\in \sV}{x\leq y}$. 
Let $S=\type (\myminus{b})$ be the type of $\myminus{b}$. 
According to Lemma~\ref{Lema1}, we have 
$\cup_{j\in J} S_j(\myminus{b})=\left\{ 1,\ldots ,p\right\}$, 
and so the half-space 
\[ 
\set{x\in \rmax^n}
{\oplus_{i\in I} c_i x_i  \leq \oplus_{j\in J} c_j x_j } 
\] 
contains $\sV$ if $\myminus{c}$ belongs to the (bounded) cell $X_S$. 

Let $\myminus{(a^{s})}\in \rmax^n$, where $s\in \left\{ 1,\ldots ,m \right\}$ 
for some $m\in \N$, be the vertices which belong to $X_S$. Then, 
by Lemma~\ref{LemmaVertexHull} we know that there exist scalars 
$\lambda_{s}$ such that $\myminus{b}=\min_{1\leq s \leq m} \lambda_{s} \myminus{(a^s)}$, 
and thus 
\[
b= \oplus_{1\leq s \leq m} \myminus{(\lambda_{s})} a^{s} \; . 
\] 
Since $y$ does not belong to $\sH$ we have 
\[
\oplus_{i\in I} \left( \oplus_{1\leq s \leq m} \myminus{(\lambda_{s})} a_i^{s} \right) y_i  =
\oplus_{i\in I} b_i y_i > 
\oplus_{j\in J} b_j y_j  = 
\oplus_{j\in J} \left( \oplus_{1\leq s \leq m} \myminus{(\lambda_{s})} a_j^{s} \right) y_j  \; ,
\]
and so there exists $r\in \left\{ 1,\ldots ,m \right\}$ such that 
\[ 
\oplus_{i\in I} \myminus{(\lambda_{r})} a_i^{r} y_i > 
\oplus_{j\in J} \myminus{(\lambda_{r})} a_j^{r} y_j \; . 
\] 
This means that the half-space 
\[ 
\set{x\in \rmax^n}
{\oplus_{i\in I} a^r_i x_i  \leq \oplus_{j\in J} a^r_j x_j} \; , 
\]
whose apex is the vertex $\myminus{(a^r)}$, separates $\sV$ from $y$.
\end{proof}

The previous proposition leads us to study 
minimal half-spaces with a fixed apex. 

\begin{lemma}\label{Lemma3}
The maximal number of incomparable half-spaces of $\rmax^n$ with 
a given apex is ${n\choose \lfloor\frac{n}{2}\rfloor}$.
\end{lemma}

\begin{proof} 
By Lemma~\ref{LemmaInclu} two half-spaces 
with the same apex $\myminus{a}\in \rmax^n$ 
\[
\sH'=\set{x\in \rmax^n}
{\oplus_{i\in I'} a_i x_i  \leq \oplus_{j\in J'} a_j x_j } 
\]
and  
\[
\sH=\set{x\in \rmax^n}
{\oplus_{i\in I} a_i x_i  \leq \oplus_{j\in J} a_j x_j } 
\]
satisfy $\sH' \subset \sH$ if, and only if, $I\subset I'$. 
Therefore, the maximal number of incomparable half-spaces with a 
given apex is equal to the maximal number of incomparable 
subsets of $\left\{ 1,\ldots , n\right\} $, which is equal to 
${n\choose \lfloor\frac{n}{2}\rfloor}$ according to  
Sperner's Theorem (see~\cite{engel97}). 
\end{proof}

\begin{remark}
There exist cones $\sV\subset \rmax^n$ which have 
${n\choose \lfloor\frac{n}{2}\rfloor}$ 
minimal half-spaces with a given apex. For example, 
consider $n$ odd and define $\sV$ as the cone generated by the following 
vectors: 
\[
v^I_i:= \left\{
\begin{array}{ll}
\zero & \makebox{ if } i\in I\; , \\
\unit & \makebox{ otherwise.}
\end{array}
\right.
\]
where $I$ is any subset of $\left\{ 1,\ldots , n\right\} $ with exactly 
$ \lfloor\frac{n}{2}\rfloor$ elements. Then, 
applying Theorem~\ref{TheoCharacMin},
it can be checked that any half-space of the form  
\[
\set{x\in \rmax^n}
{\oplus_{i\in I} x_i \leq \oplus_{j\not \in I} x_j } \; , 
\]
where again $I$ is any subset of $\left\{ 1,\ldots , n\right\} $ with exactly 
$ \lfloor\frac{n}{2}\rfloor$ elements, is minimal with respect to $\sV$. 

When the generators of a cone $\sV\subset \rmax^n$ are in general position, 
it is possible to have (at least) $P(n)$ minimal half-spaces with respect to 
$\sV$ with the same apex, where $\{ P(n) \}_{n\in \N}$ is the Padovan sequence, 
which is defined by the recurrence 
\begin{align}\label{e-padovan}
P(n)=P(n-2)+P(n-3)
\end{align}
with $P(1)=P(2)=P(3)=1$. More precisely, the max-plus cone 
$\sV \subset \rmax^n$ generated by the vectors $v^r=(1 r,2 r, \ldots ,n r)^T$ 
for $r=1,\ldots ,n$, where the product is in the usual algebra, 
has $P(n)$ minimal half-spaces with apex $a$, 
where $a_i:=\sum_{k=1}^i k$ for $i=1,\ldots ,n$. 

To see this, 
in the first place note that due to the definition of $a$, 
the type of $a$ is given by $S_n(a)=\{ n \}$ 
and $S_k(a)=\{ k,k+1 \}$ for $1\leq k <n$. 
Since by Lemmas~\ref{LemmaInclu} and~\ref{Lema1} 
minimal half-spaces with a fixed apex $a$ correspond to 
subsets $J\subset \{1,\ldots ,n\}$ such that $\{S_j(a)\}_{j\in J}$ 
is a minimal covering of $\{1,\ldots ,n\}$, 
it follows that each time $S_{r-1}(a)$ and $S_{r}(a)$ 
belong to such a covering, 
then $S_{r+1}(a)$ cannot belong to 
it because $S_r(a)\subset S_{r-1}(a)\cup S_{r+1}(a)$. 
Observe also that if  $S_{r-1}(a)$ belongs to a minimal covering but 
$S_r(a)$ does not, then $S_{r+1}(a)$ must belong to it. 
Finally, since $S_n(a)\subset S_{n-1}(a)$, precisely one of these 
two sets must belong to a minimal covering. 

Let $\{S_j(a)\}_{j\in J}$ be a minimal covering of $\{1,\ldots ,n\}$,   
and assume that $S_n(a)$ belongs to it. Then, 
$S_{n-2}(a)$ must also belong to the covering.  
If we define the sets $S'_j(a)$ for $j=1,\ldots ,n-3$ by 
$S'_{n-3}(a):=\{n-3\}$ and $S'_{k}(a):=S_{k}(a)$ for $1\leq k < n-3$, 
then it can be checked 
that there is a bijection between minimal coverings of $\{1,\ldots ,n-3\}$ 
by the sets $S'_j(a)$ and minimal coverings of $\{1,\ldots ,n\}$ by the 
sets $S_j(a)$ which contain $S_n(a)$. 

Analogously, if we now assume that $S_{n-1}(a)$ belongs to a minimal covering 
of $\{1,\ldots ,n\}$, 
and if we define the sets $S''_j(a)$ for $j=1,\ldots ,n-2$ by 
$S''_{n-2}(a):=\{n-2\}$ and $S''_{k}(a):=S_{k}(a)$ for $1\leq k < n-2$, 
it can be checked that there is a bijective correspondence between minimal 
coverings of $\{1,\ldots ,n-2\}$ 
by the sets $S''_j(a)$ and minimal coverings of $\{1,\ldots ,n\}$ by the 
sets $S_j(a)$ which now contain $S_{n-1}(a)$. 

In consequence, the number of minimal coverings of $\{1,\ldots ,n\}$ 
by the sets $S_j(a)$ is given by the Padovan sequence because 
these numbers satisfy the recurrence relation that defines this sequence.  
\end{remark} 

\begin{example} 
Taking $n=4$, the argument used to establish the recurrence~\eqref{e-padovan}
defining the Padovan sequence shows that at $a=(1,3,6,10)^T$,
there are two minimal coverings of $\{1,\ldots ,4\}$ by the sets $S_j(a)$. 
One consists of $S_1(a)=\{1,2\}$ and $S_3(a)=\{3,4\}$, 
and corresponds to the half-space
\[
(-3)x_2\oplus (-10)x_4 \leq (-1)x_1\oplus (-6)x_3 
\]
which coincides with the one in~\eqref{e-general} when $\delta=-1$
and has the same shape as the ones in Figure~\ref{fig-minimal}. The 
second minimal covering consists of
$S_1(a)=\{1,2\}$, $S_2(a)=\{2,3\}$, $S_4(a)=\{4\}$,
it corresponds to the half-space
\[
(-6)x_3  \leq (-1)x_1 \oplus (-3)x_2\oplus (-10)x_4
\]
which is represented in Figure~\ref{fig-padovan}.
\begin{figure}
\begin{center}
\includegraphics[scale=0.33]{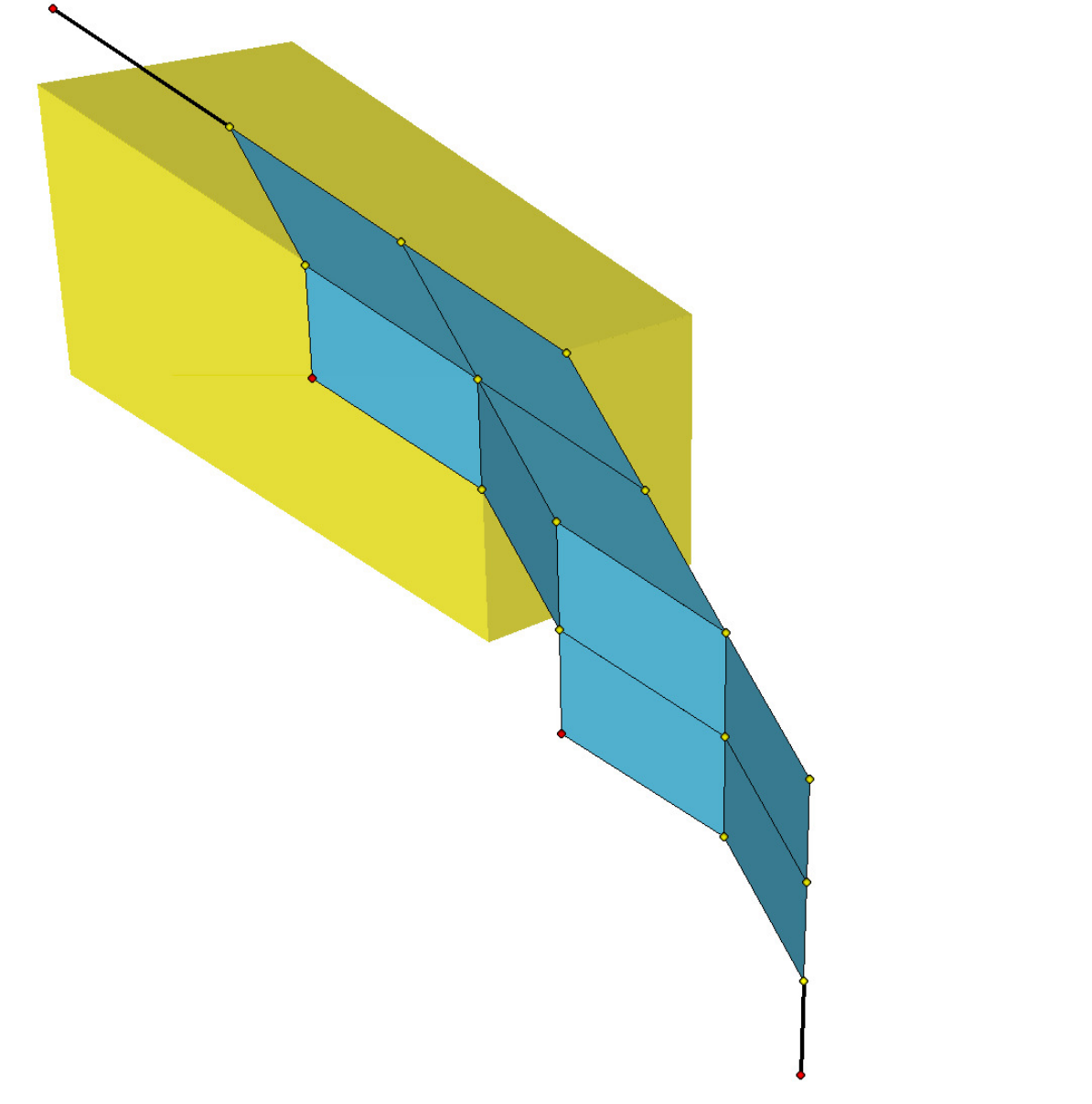}
\end{center}
\caption{One of the two minimal half-spaces with apex $(1,3,6,10)^T$.}\label{fig-padovan}
\end{figure}
\end{example}
 
\section{Relation between the extreme rays of the polar and minimal half-spaces}

In the classical theory of convex cones, 
it is known that the extreme rays of the polar of 
a convex cone correspond to its supporting half-spaces. 
Since the notion of extreme rays carries over to the max-plus 
setting~\cite{BSS,gk07} as well as the notion of polar~\cite{katz08}, 
it is natural to investigate the relation between the minimal half-spaces 
with respect to a max-plus cone and the extreme rays of its polar.  

Following~\cite{katz08}, 
we define the {\em polar} of a max-plus cone $\sV\subset \rmax^n$ as 
\[
\sV^\circ :=\set{(a,b)\in (\rmax^n)^2}{\oplus_{1\leq i\leq n} a_i x_i 
\leq \oplus_{1\leq j\leq n} b_j x_j ,\; \forall x\in \sV} \; , 
\]
i.e.\ $\sV^\circ $ represents the set of all 
the half-spaces which contain $\sV$. Conversely, 
we may consider the max-plus cone defined by the 
intersection of a set of half-spaces. 
This leads to define (see~\cite{katz08}),
for all $\sW\subset (\rmax^n)^2$, a ``dual'' polar cone
\[
\sW^\diamond := \set{x\in \rmax^n}{\oplus_{1\leq i\leq n} a_i x_i  
\leq \oplus_{1\leq j\leq n} b_j x_j  ,\; \forall (a,b)\in \sW} \; .
\]
Then, by the separation theorem for closed cones 
(\cite{zimmerman77,shpiz,cgqs04}), 
it follows that a closed cone $\sV$ is characterized by its polar cones:
\[
\sV = (\sV^\circ)^\diamond \; .
\]
In particular, when $\sV$ is finitely generated, 
this means that $\sV=\sW^\diamond$, 
where $\sW\subset (\rmax^n)^2$ is the (finite) 
set of extreme vectors of $\sV^\circ$. Thus, the extreme
vectors of the polar of $\sV$ determine a finite family
of max-plus linear inequalities defining $\sV$. 

The following theorem characterizes the extreme vectors 
of $\sV^\circ$ in terms of the generators of $\sV$. 

\begin{theorem}\label{TheoCharacExtreme}
Assume that $\sV\subset \rmax^n$ is a max-plus cone with full support 
generated by the vectors $v^{r} \in \rmax^n$, where $r=1,\ldots , p$. Then, 
up to a non-zero scalar multiple, 
the extreme vectors of $\sV^\circ $ are either $(\zero,\uvector^i)$ or 
$(\uvector^i,\uvector^i)$, for $i=1,\ldots ,n$, or have the form 
$(\uvector^i, \oplus_{j\in J} b_j \uvector^j)$
for some $i\in \{1,\dots ,n\}$, 
where $J\subset \left\{ 1,\ldots ,n\right\}\setminus \{ i \}$.  
Moreover, a vector of $\sV^\circ $ of the latter form  
is extreme if, and only if, 
the following condition is satisfied:
\begin{equation}\label{CondExtrePolar}
\makebox{For each } j\in J \makebox{ there exists } 
r\in \left\{ 1,\ldots ,p\right\} \makebox{ such that } 
v^r_i=b_j v_j^r> \oplus_{k\in J\setminus \{ j \}} b_k v^r_k 
 \; .
\end{equation}
\end{theorem}

\begin{proof} 
In the first place, note that by the definition of $\sV^\circ $, 
the vectors $(\zero ,\uvector^i)$, for $i=1,\ldots ,n$, 
belong to $\sV^\circ $,  
so these vectors are clearly extreme vectors of $\sV^\circ$ 
and the only ones of the form $(\zero ,b)$. 
Moreover, since $\sV$ has full support, $\sV^\circ $ does not contain 
vectors of the form $(\uvector^i,b_i\uvector^i)$ with $b_i<\unit$. 
This implies that $(\uvector^i,\uvector^i)$ is also 
an extreme vector of $\sV^\circ $ for any $i=1,\ldots ,n$.   

Since $\sV^\circ$ satisfies 
\[
(a'\oplus a'',b)\in \sV^\circ \implies (a',b)\in \sV^\circ \makebox{ and } 
(a'',b)\in \sV^\circ \; ,
\] 
it follows that $(a,b)$ is an extreme vector of $\sV^\circ$ 
with $a\neq \zero $ only if there exists 
$i\in \left\{ 1,\ldots ,n\right\}$ such that 
either $\sop (a)= \{ i \}\not \subset \sop (b)$ or 
$(a,b)$ is a non-zero scalar multiple of $(\uvector^i,\uvector^i)$. 
Therefore, in the former case we may assume that $a=\uvector^i$ 
for some $i\in \left\{ 1,\ldots ,n\right\}\setminus \sop (b)$. 

Let $(\uvector^i, \oplus_{j\in J} b_j \uvector^j)$, 
with $i\not \in J\subset \{1,\ldots , n \}$, be a vector of $\sV^\circ$ 
which satisfies Condition~\eqref{CondExtrePolar}. 
Assume that $(\uvector^i, \oplus_{j\in J} b_j \uvector^j)=
\oplus_{1\leq s\leq m} (a^s,b^s)$, where $m\in \N$ and
$(a^s,b^s)\in \sV^\circ$ for all $s\in \{1,\ldots ,m\}$. Then, 
there exists $l\in \{1,\ldots ,m\}$ such that $a^l = \uvector^i$. 
We claim that 
$b^l =\oplus_{j\in J} b_j \uvector^j$, which implies that 
$(\uvector^i, \oplus_{j\in J} b_j \uvector^j)$ is an extreme vector of 
$\sV^\circ $. By the contrary, assume that 
$b^l \neq \oplus_{j\in J} b_j \uvector^j$. 
Then, since $b^l \leq \oplus_{j\in J} b_j \uvector^j$, 
we must have $b_j^l< b_j$ for some $j\in J$. 
By Condition~\eqref{CondExtrePolar} for this $j\in J$ 
there exists $r\in \left\{ 1,\ldots ,p\right\}$ 
such that $v^r_i=b_j v_j^r> (\oplus_{k\in J\setminus \{ j \}} b_k v^r_k)$, 
and thus 
\[
\oplus_{1\leq h\leq n} a_h^l v_h^r=v^r_i=b_j v_j^r>
(\oplus_{k\in J\setminus \{ j \}} b_k v^r_k)\oplus  b_j^l v_j^r 
\geq \oplus_{1\leq k\leq n} b_k^l v^r_k \; , 
\]
which contradicts the fact that $(a^l ,b^l )\in \sV^\circ$. 
This proves the ``if'' part of the second statement of the theorem. 

Now assume that $(\uvector^i, \oplus_{j\in J} b_j \uvector^j)$ 
is an extreme vector of $\sV^\circ$. 
If Condition~\eqref{CondExtrePolar} was not satisfied, 
there would exist $j\in J$ and $\delta < 0$ such that 
\[
v_i^r \leq (\oplus_{k\in J\setminus \{ j \}} b_k v^r_k)\oplus \delta b_j v_j^r 
\]
for all $r\in \left\{ 1,\ldots ,p\right\}$, implying that 
$(\uvector^i,(\oplus_{k\in J\setminus \{ j \} } b_k \uvector^k)
\oplus \delta b_j \uvector^j ) \in \sV^\circ $. Then, 
we would have 
\[
(\uvector^i, \oplus_{j\in J} b_j \uvector^j) = 
(\uvector^i, (\oplus_{k\in J\setminus \{ j \} } b_k \uvector^k) 
\oplus \delta b_j \uvector^j ) \oplus (\zero , b_j \uvector^j) \;
\]
which contradicts the fact that 
$(\uvector^i, \oplus_{j\in J} b_j \uvector^j)$ 
is extreme because $(\zero , b_j \uvector^j) \in \sV^\circ $. 
This completes the proof of the theorem.  
\end{proof} 
More generally, there is a hypergraph characterization of the extreme
points of a max-plus cone defined by finitely many linear inequalities~\cite{AGK09}. In the special case of the polar, Theorem~\ref{TheoCharacExtreme} shows that this hypergraph reduces to a star-like graph.

The following proposition shows that the
extreme vectors of the polar $\sV^\circ$ are special minimal half-spaces,
up to a projection of $\sV$.
Here, $\rmax^{J\cup \{i\}}$ 
denotes the vectors obtained by keeping only the entries of vectors 
of $\rmax^n$ whose indices belong to the set $J\cup \{i\}$.
\begin{proposition}
A vector $(\uvector^i, \oplus_{j\in J} b_j \uvector^j)$ 
of the polar $\sV^\circ $ is extreme if, and only if, 
\[
\set{x\in \rmax^{J\cup \{i\} } }{x_i  \leq \oplus_{j\in J} b_j x_j}
\] 
is a minimal half-space with respect to the projection of $\sV$ on 
$\rmax^{J\cup \{i\} }$.
\end{proposition}
\begin{proof}
This follows readily from Theorem~\ref{TheoCharacExtreme} and 
Lemma~\ref{LemmaInclu}.
\end{proof}
\begin{remark}
When the entries of the generators of $\sV$ 
are all finite, if $(\uvector^i, \oplus_{j\in J} b_j \uvector^j)$ 
is an extreme generator of $\sV^\circ $, 
Proposition~17 of~\cite{DS04} and Condition~\eqref{CondExtrePolar} imply that the 
projection of $(\oplus_{j\in J} \myminus{(b_j)} \uvector^j)\oplus \uvector^i $ 
on $\rmax^{J\cup \{i\} }$ is a vertex of the natural cell decomposition of 
$\rmax^{J\cup \{i\} }$ induced by the projection of the generators of 
$\sV$ on $\rmax^{J\cup \{i\}}$. 
\end{remark}

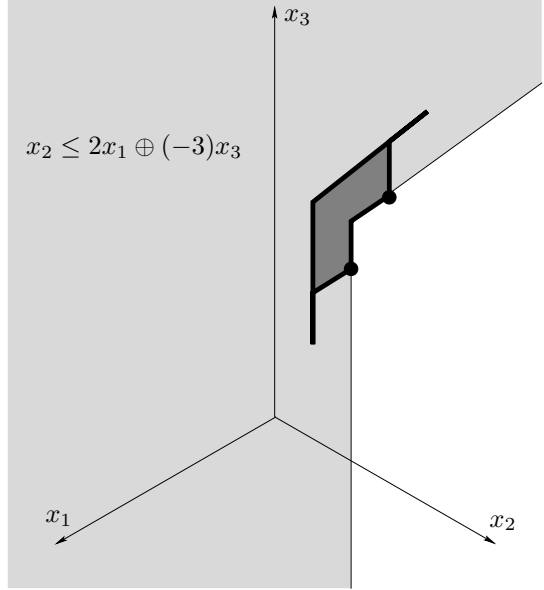
\begin{figure}
\begin{center}
\begin{picture}(0,0)%
\includegraphics{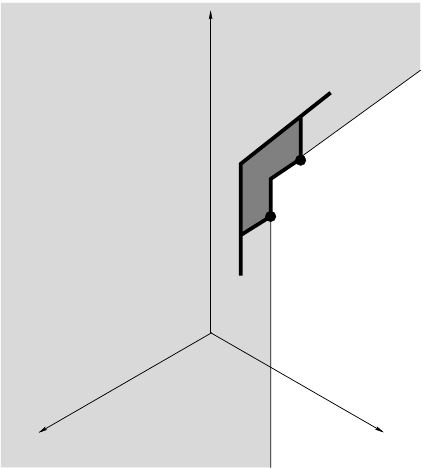}%
\end{picture}%
\setlength{\unitlength}{1579sp}%
\begingroup\makeatletter\ifx\SetFigFont\undefined%
\gdef\SetFigFont#1#2#3#4#5{%
  \reset@font\fontsize{#1}{#2pt}%
  \fontfamily{#3}\fontseries{#4}\fontshape{#5}%
  \selectfont}%
\fi\endgroup%
\begin{picture}(8434,9334)(1789,-6683)
\put(2401,-5611){\makebox(0,0)[lb]{\smash{{\SetFigFont{10}{12.0}{\rmdefault}{\mddefault}{\updefault}{\color[rgb]{0,0,0}$x_1$}%
}}}}
\put(6151,2264){\makebox(0,0)[lb]{\smash{{\SetFigFont{10}{12.0}{\rmdefault}{\mddefault}{\updefault}{\color[rgb]{0,0,0}$x_3$}%
}}}}
\put(2101,164){\makebox(0,0)[lb]{\smash{{\SetFigFont{10}{12.0}{\rmdefault}{\mddefault}{\updefault}{\color[rgb]{0,0,0}$x_2 \leq 2 x_1 \oplus (-3) x_3$}%
}}}}
\put(9376,-5686){\makebox(0,0)[lb]{\smash{{\SetFigFont{10}{12.0}{\rmdefault}{\mddefault}{\updefault}{\color[rgb]{0,0,0}$x_2$}%
}}}}
\end{picture}%
\end{center}
\caption{Illustration of the relation between extreme vectors of the polar and minimal half-spaces. The minimal half-space is in light gray. The two support vectors are represented by bold points.}
\label{figureExtMin} 
\end{figure}

\begin{example}
Consider again the max-plus cone $\sV\subset \rmax^4$ of 
Example~\ref{Ejemplo1}. Applying Theorem~\ref{TheoCharacExtreme}, 
it can be checked that $(\uvector^2,2\uvector^1\oplus (-3)\uvector^3 )$  
is an extreme vector of $\sV^\circ $. The projection of $\sV$ on 
$\rmax^{ \{1,2,3\} }$ is represented in Figure~\ref{figureExtMin} 
by the bounded dark gray region together with the two line segments 
joining the points $(0,1,2)^T$ and $(0,4,8)^T$ to it. 
The unbounded light gray region represents the projection of the half-space 
$\set{x\in \rmax^4}{x_2  \leq 2 x_1\oplus (-3) x_3}$. 
The fact that this projection is minimal with respect to the projection 
of $\sV$ is geometrically clear from the figure. 
\end{example}

\begin{remark}
Condition~\eqref{CondExtrePolar} of Theorem~\ref{TheoCharacExtreme} 
shows that when $(\uvector^i, \oplus_{j\in J} b_j \uvector^j)$ is an
extreme vector of the polar $\sV^\circ $, the hyperplane
\[
\sH^{=}=\set{x\in \rmax^n}{x_i = \oplus_{j\in J} b_j x_j}
\]
contains at least $|J|$ generators $v^r$ of $\sV$. The latter
may be thought of as {\em support vectors}. It also follows
from this theorem that the coefficients
$b_j$ of this hyperplane are uniquely determined by these support vectors.
\end{remark}
\begin{remark}
We noted above that the set $\sW$ of extreme vectors of the polar $\sV^\circ$ 
satisfies $\sW^\diamond =\sV$, in other words, it yields
a finite family of max-plus linear inequalities defining $\sV$,
the size of which can be bounded by using the results of~\cite{AGK09}. However,
the bipolar theorem of~\cite{katz08}
shows that $\sW$ is not always a minimal set with this property.
\end{remark}

\end{document}